\documentclass[11pt, reqno]{amsart}%

\usepackage{amssymb}
\usepackage{amsmath,esint}
\usepackage[shortlabels]{enumitem}
\usepackage{amsthm, dsfont}
\usepackage{amscd,mdwlist}
\usepackage{color}
\usepackage{cite}
\usepackage{bbm}
 \usepackage{tcolorbox}
 \usepackage{thmtools}
 \usepackage[backgroundcolor=white, bordercolor=blue,
 linecolor=blue]{todonotes}
 \usepackage{relsize}
\usepackage[pdfpagelabels, plainpages=false]{hyperref}

\setlist[description]{leftmargin=\parindent,labelindent=\parindent}
\setlist[enumerate]{leftmargin=8mm ,labelindent=0cm}

\hypersetup{
 colorlinks=false,
 linkcolor=red,%---------------------%couleur des liens dans le document (references
 filecolor=magenta,%-----------------%couleur des files à ouvrir
 urlcolor=cyan,%---------------------%couleur des liens internets à ouvrir
 citecolor=blue,%--------------------%couleur des liens pour les bibliographies
}%-------------------------------------%pour les courleur des liens ; ne jamias oublier de mettre le package hyperref avant
 
 \usepackage[
 paper=a4paper,
 left=2cm,
 right=25mm,
 top=2.5cm,
 bottom=2cm,
 includefoot,
 foot=8mm,
 bindingoffset=0mm]{geometry}

\theoremstyle{plain}
\declaretheorem[title=Theorem, parent=section]{theorem}
\declaretheorem[title=Lemma,sibling=theorem]{lemma}
\declaretheorem[title=Proposition,sibling=theorem]{proposition}
\declaretheorem[title=Corollary,sibling=theorem]{corollary}

\theoremstyle{definition}
\declaretheorem[title=Definition,sibling=theorem]{definition}
\declaretheorem[title=Remark,sibling=theorem]{remark}
\declaretheorem[title=Remark, numbered=no]{remark*}
\declaretheorem[title=Example, sibling=theorem]{example}

\numberwithin{equation}{section}

\DeclareMathOperator{\dist}{dist}
\DeclareMathOperator{\supp}{supp}

\newtheorem{innercustomthm}{Theorem}
\newenvironment{customthm}[1]
{\renewcommand\theinnercustomthm{#1}\innercustomthm}
{\endinnercustomthm}

\newtheorem{counterexample}[theorem]{Counterexample}

\newcommand{\il}{\int\limits}
\newcommand{\iil}{\iint\limits}
\DeclareMathOperator{\R}{\mathbb{R}}
\renewcommand{\d}{\,\mathrm{d}}
\renewcommand{\div}{\operatorname{div}}

\newcommand{\WnuOmR}{W_{\nu}^{p}(\Omega|\R^d)}
\newcommand{\WnuOm}{W_{\nu}^{p}(\Omega)}

\newcommand{\eps}{\varepsilon}
\newcommand{\vertiii}[1]{{\left\vert\kern-0.25ex\left\vert\kern-0.25ex\left\vert #1 \right\vert\kern-0.25ex\right\vert\kern-0.25ex\right\vert}}

\usepackage[english]{babel}
\addto\extrasenglish
{%

}

\makeatletter
\@namedef{subjclassname@2020}{\textup{2020} Mathematics Subject Classification}
\makeatother

\title{A remake of Bourgain-Brezis-Mironescu characterization of Sobolev spaces 
\\\tiny{\url{https://doi.org/10.1007/s42985-023-00232-4}}}

\author{Guy Fabrice Foghem Gounoue}
\address{{\tiny Technische Universt\"{a}t Dresden, Fakult\"{a}t f\"{u}r Mathematik, Institut f\"{u}r Wissenschaftliches Rechnen,  Zellescher Weg 23-25 01217, Dresden, Germany. E-mail: guy.foghem[at]tu-dresden.de}} 
\thanks{Financial support by the DFG via IRTG 2235: ``Searching for the regular in the irregular: Analysis of singular and random systems'' is gratefully acknowledged.}
\thanks{Financial support by the DFG via the Research Group 3013: ``Vector-and Tensor-Valued Surface PDEs'' is gratefully acknowledged.}

\begin{document}
\begin{abstract}
We introduce a large class of concentrated $p$-L\'{e}vy integrable functions approximating the unity, which serves as the core tool from which we provide a nonlocal characterization of Sobolev spaces and the space of functions of bounded variation via nonlocal energies forms. It turns out that this nonlocal characterization is a necessary and sufficient criterion to define Sobolev spaces  on domains satisfying the extension property. We also examine the  general case where the extension property does not necessarily hold. In the latter case we establish weak convergence of the nonlocal Radon measures involved to the local Radon measures induced by the distributional gradient.
\end{abstract}

\keywords{Nonlocal energy forms, $p$-L\'{e}vy integrability, Sobolev spaces, Bounded variation spaces, Extension domain}

\subjclass[2020]
{26B30, %Absolutely continuous functions, functions of bounded variation
46B45,  % Banach sequence spaces
%46A45, %sequence space
46E27,  %Spaces of measures, convergence of measures
46E30,  %Spaces of measurable functions (Lp-spaces, Orlicz spaces, Lorentz spaces, etc.)
46E35% Sobolev spaces and other spaces of “smooth” functions, embedding theorems, trace theorems
%49J40, %Variational inequalities
%60F25 % Lp-limit theorems
}

\maketitle
\date{\today}
%\tableofcontents
\section{Introduction} 

Let $\Omega$ be an open subset of $\R^d$, $d\geq 1$,  and $1\leq p<\infty$. We aim to provide a nonlocal characterization of first order Sobolev spaces on $\Omega$ using the following type nonlocal energy forms
\begin{align}\label{eq:nonlocal-form}
\mathcal{E}^i_{\Omega}(u)= \iil_{ \R^d \R^d}|u(x)- u(y)|^pa^i_\Omega(x,y)\nu(x-y)\d y\d x, \quad i=1,2,3, 
\end{align}
where, $\nu:\R^d\setminus\{0\}\to [0,\infty)$ is measurable and satisfies the $p$-L\'{e}vy integrability condition 
\begin{align}\label{eq:plevy-cond}
\int_{\R^d}(1\land |h|^p)\nu(h) \d h<\infty\,,
\end{align} 
and $a^1_\Omega(x,y)= \min(\mathds{1}_{\Omega}(x), \mathds{1}_{\Omega} (y))$, $a^2_\Omega(x,y)= \max(\mathds{1}_{\Omega}(x), \mathds{1}_{\Omega} (y))$
%= 1-\min(\mathds{1}_{\Omega^c}(x), \mathds{1}_{\Omega^c} (y))$ 
and $a^3_\Omega(x,y)= \frac12(\mathds{1}_{\Omega}(x)+\mathds{1}_{\Omega} (y))$; where $\mathds{1}_{\Omega}$ is the indicator function of $\Omega$. Here and in what follows, the notation $a\land b$ stands for $ \min(a,b)$, $a,b\in \R$ . 
More explicitly, we can write the forms $\mathcal{E}^i_{\Omega},$ as  follows
\begin{align*}
\mathcal{E}^1_{\Omega}(u)&= \iil_{\Omega\Omega}|u(x)- u(y)|^p\nu(x-y)\d y\d x,
\\
\mathcal{E}^2_{\Omega}(u)&= \iil_{\mathcal{G}(\Omega)}|u(x)- u(y)|^p\nu(x-y)\d y\d x, \quad (\mathcal{G}(\Omega)= (\R^d\times \R^d)\setminus(\Omega^c\times\Omega^c)), 
\\
\mathcal{E}^3_{\Omega}(u)&= \iil_{\Omega \R^d}|u(x)- u(y)|^p\nu(x-y)\d y\d x.
\end{align*}
Note in passing that $\mathcal{E}^1_{\Omega}\leq \mathcal{E}^i_{\Omega},$ $\mathcal{E}^1_{\R^d}=\mathcal{E}^2_{\R^d}=\mathcal{E}^3_{\R^d}$ and   $\frac12\mathcal{E}^2_{\Omega} \leq \mathcal{E}^3_{\Omega}\leq \mathcal{E}^2_{\Omega}$ since $\frac12a^2_{\Omega}\leq a^3_{\Omega}\leq a^2_{\Omega}$. The nonlocal forms $\mathcal{E}^i_{\Omega}$ are crucial in the study of Integro-Differential Equations (IDEs) involving nonlocal operators of $p$-L\'{e}vy types; see for instance the recent works \cite{Fog20,FK22, DFK22}. For $p=2$, \eqref{eq:plevy-cond}, is the well-known L\'{e}vy integrability condition. Actually, when $\nu$ is radial, the $p$-L\'{e}vy integrability \eqref{eq:plevy-cond} condition is consistent and self-generated in the sense that condition \eqref{eq:plevy-cond} holds true if and only if $\mathcal{E}^1_{\R^d}(u)<\infty$ for all $u\in C_c^\infty(\R^d)$; see Section \ref{sec:approx-dirac-prop} for the details. In addition, the $p$-L\'{e}vy integrability condition \eqref{eq:plevy-cond} indicates that $\nu$ is allowed to have a heavy singularity at the origin. For instance, $\nu(h)=|h|^{-d-sp}$ satisfies the condition \eqref{eq:plevy-cond} if and only if $s\in (0,1)$. 

 Next, to reach our goal, we need to introduce a general class of approximation of the unity  by $p$-L\'{e}vy integrable functions. To be more precise, our standing approximation tool consists of a family of $p$-L\'{e}vy integrable functions, $(\nu_\eps)_\eps$ satisfying, for each $\eps>0$ and every $\delta>0$, 
 \begin{align}\label{eq:plevy-approx}
 \hspace{-0.2ex}\nu_\eps\geq0\,\, \text{is radial},\, \int_{\R^d}\hspace{-0.1ex}(1\land |h|^p)\nu_\eps(h) \d h=1\quad \text{and}\quad \lim_{\varepsilon\to 0}\hspace{-0.2ex}\int_{|h|>\delta} \hspace{-1ex}(1\land |h|^p)\nu_\eps(h)\d h=0. 
 \end{align}
 \noindent For instance, assume $\nu$ is radial and $\int_{\R^d}(1\land |h|^p)\nu(h) \d h=1$ then (see Proposition \ref{prop:rescalled}) one obtains a remarkable family $(\nu_\eps)_\eps$ satisfying \eqref{eq:plevy-approx} by the following rescaling
\begin{align*}%\label{eq:rescalled-levy-measure}
\begin{split}
\nu_\varepsilon(h) = 
\begin{cases}
\varepsilon^{-d-p}\nu\big(h/\varepsilon\big)& \text{if}~~|h|\leq \varepsilon,\\
\varepsilon^{-d}|h|^{-p}\nu\big(h/\varepsilon\big)& \text{if}~~\varepsilon<|h|\leq 1,\\
\varepsilon^{-d}\nu\big(h/\varepsilon\big)& \text{if}~~|h|>1.
\end{cases}
\end{split}
\end{align*}
Another  sub-class of $(\nu_\eps)_\eps$ satisfying \eqref{eq:plevy-approx} is obtained by putting $\nu_\eps(h)= c_\eps|h|^{-p} \rho_\eps(h),$ where $(\rho_\varepsilon)_\varepsilon$ is a family of integrable functions approximating the unity, i.e., for each $\eps>0$ and every $\delta>0$, 
\begin{align}\label{eq:approx-dirac}
\rho_\eps\geq0\,\, \text{ is radial},\quad \int_{\R^d}\rho_\eps(h) \d h=1\quad \text{and}\quad \lim_{\varepsilon\to 0}\int_{|h|>\delta} \hspace{-2ex}\rho_\eps(h)\d h=0, \, 
\end{align} 
 and $c_\eps>0$ is a suitable norming constant for which the integrability condition in \eqref{eq:plevy-approx} is verified. From this perspective, the class of approximation of the unity $(\rho_\eps)_\eps$ satisfying \eqref{eq:approx-dirac} can be viewed as a subclass 
$(\nu_\eps)_\eps$ satisfying \eqref{eq:plevy-approx}. However,
the converse is not warranted. In other words the class $(\nu_\eps)_\eps$ is more general than the class $(\rho_\eps)_\eps$. This is because the family $(\nu_\eps)_\eps$ also includes families of the forms $(c_\eps|h|^{-p}\rho_\eps(h))_\eps$ for which $\rho_\eps{'s}$ are not integrable. For a simple example, consider $\nu_\varepsilon(h) = a_{\varepsilon, d,p} |h|^{-d-(1-\varepsilon)p}$ (see Example \ref{Ex: stable-class}) then $(\nu_\varepsilon)_\eps$ satisfies \eqref{eq:plevy-approx} but there is no family $(\rho_\eps)_\eps$ satisfying \eqref{eq:approx-dirac} such that $\nu_\eps(h)= c_\eps|h|^{-p}\rho_\eps(h)$.  
Viewed in the sense of the correspondence $(\rho_\eps)_\eps \mapsto (\nu_\eps)_\eps$ with $\nu_\eps= c_\eps |h|^{-p}\rho_\eps,$ the class of $(\nu_\eps)_\eps$ is therefore strictly lager than that of $(\rho_\eps)_\eps$.

\smallskip 

 We emphasize that our main goal is to characterize Sobolev spaces on an open set $\Omega\subset \R^d$ using a sequence $(\nu_\eps)_\eps$ satisfying \eqref{eq:plevy-approx}. 
 Let us recall that the Sobolev space $W^{1,p}(\Omega)$ is the Banach space of functions $u\in L^p(\Omega)$ whose first order distributional derivatives belong to $L^p(\Omega)$, with the norm $\|u\|_{W^{1,p}(\Omega)} = (\|u\|^p_{L^{p}(\Omega)}+ \|\nabla u\|^p_{L^p(\Omega)})^{1/p}.$ Another space of particular interest, that emerges naturally as a generalization of $W^{1,1}(\Omega)$, is the so called space of bounded variations $BV(\Omega)$. The space $BV(\Omega)$ consists in functions $u\in L^1(\Omega)$ with bounded variation, i.e., $|u|_{BV(\Omega) }<\infty$ where
\begin{align}\label{eq:bounded-variation}
|u|_{BV(\Omega) }:= \sup\Big\lbrace \int_{\Omega} u(x) \operatorname{div} \phi(x) \d x:~\phi \in C_c^\infty(\Omega, \R^d),~\|\phi\|_{L^\infty(\Omega, \R^d)}\leq 1\Big\rbrace.
\end{align}

\noindent The space $BV(\Omega)$ is a Banach space under the norm 
$\|u\|_{BV(\Omega) }= \|u\|_{L^1(\Omega) } + |u|_{BV(\Omega) }.$ We denote the distributional derivative of a function $u\in BV(\Omega)$ by $\nabla u$. Roughly speaking, $\nabla u =(\Lambda_1, \Lambda_2,\cdots, \Lambda_d)$ is a vector valued Radon measure on $\Omega$ such that 
\begin{align*}
\int_\Omega u(x)\frac{\partial \varphi}{\partial x_i}(x) \d x=
-\int_\Omega \varphi(x)\d \Lambda_i(x), \quad \text{for all }\quad \varphi\in C_c^\infty(\Omega),~~i=1,\cdots, d.
\end{align*}
 The quantity $|\nabla u|= (\Lambda_1^2+\cdots+\Lambda_d^2)^{1/2}$ is a positive Radon measure whose value on an open set $U\subset \Omega$ is $ |\nabla u|(U) =|u|_{BV(U)}$. 
 Conventionally, we put $\|\nabla u\|_{L^p(\Omega)} = \infty$ if $|\nabla u|$ is not in $L^p(\Omega)$ with $1<p<\infty$ and for $p=1$, $| u|_{B V(\Omega)} = \infty$ if the measure $|\nabla u|$ does not have a finite total variation. Note that, if $u\in W^{1,1}(\Omega)$ then $u\in BV(\Omega)$, $\partial_{x_i} u(x)\d x= \d\Lambda_i(x)$ and $ |u|_{BV(\Omega)}= \|\nabla u\|_{L^1(\Omega)}.$ Indeed, since $u\in W^{1,1}(\Omega)$, the integration by part implies 
\begin{align*}%\label{eq:bounded-variation}
|u|_{BV(\Omega) }= \sup\Big\lbrace \int_{\Omega}\nabla u(x) \cdot \phi(x) \d x:~\phi \in C_c^\infty(\Omega, \R^d),~\|\phi\|_{L^\infty(\Omega, \R^d)}\leq 1\Big\rbrace\leq \|\nabla u\|_{L^1(\Omega)}.
\end{align*}
%\begin{alignat*}{2}
%|u|_{BV(\Omega)} &= \sup_{\|\phi\|_{L^{\infty}(\Omega,\R^d)}\leq 1} \Big|\int_{\Omega}\nabla u(x) \, \cdot \phi(x)~\mathrm{d}x\Big|
%%\leq \sup_{\|\phi\|_{L^{\infty}(\Omega, \R^d)}\leq 1} \int_{\Omega} |\nabla u(x)| |\phi(x)|~\d x
%\leq \|\nabla u\|_{L^1(\Omega)}, 
%\end{alignat*}
%with supremum over $\phi\in C_c^\infty(\Omega,\R^d)$. 
Conversely, since $\partial_{x_i}u\in L^1(\Omega)$, take a sequence $(\chi_{n})_n\subset C_c^\infty(\Omega,\R^d)$, converging to $\nabla u$ in $L^1(\Omega, \R^d)$ and a.e. in $\Omega$. Define $\chi_n^\eps\in C_c^\infty(\Omega, \R^d)$, $\eps>0$ by $\chi^\eps_n= \chi_n(|\chi_n|^2+\eps^2)^{-1/2}$, so that $\|\chi_n^\eps\|_{L^{\infty}(\Omega, \R^d)} \leq 1$. The convergence dominated theorem and the integration by parts imply that 
\begin{align*}
\int_{\Omega} |\nabla u(x)|^2 (|\nabla u(x)|^2+\eps^2)^{-1/2}\d x
%&= \lim_{n\to \infty}\Big|\int_{\Omega} \nabla u(x)\cdot \chi^\eps_n (x)\d x \Big|\\
&=\lim_{n\to \infty}\Big|\int_{\Omega} u(x) \div \chi^\eps_n(x)\d x\Big|
\leq |u|_{BV(\Omega)}. 
\end{align*}

\noindent Whence Fatou's lemma  implies 
\begin{align*}
\int_{\Omega} |\nabla u(x)|\d x\leq \liminf_{\varepsilon\to 0} \int_{\Omega} |\nabla u(x)|^2 (|\nabla u(x)|^2+\eps^2)^{-1/2}\d x\leq |u|_{BV(\Omega)}. 
\end{align*}

\bigskip 

 We are now in position to state our first result.
\begin{theorem}\label{thm:liminf}
Let $\Omega\subset \R^d$ be open and $u\in L^p(\Omega)$ such that 
\begin{align}\label{eq:liminf-condition}
A_p:= \liminf_{\varepsilon\to 0} \iil_{\Omega\Omega } |u(x)-u(y)|^p\nu_\varepsilon(x-y)\d y\d x <\infty\,.
\end{align}
Then $u\in W^{1,p}(\Omega) $ for $1<p<\infty$ and $u\in BV(\Omega)$ for $p=1$. Moreover, there hold the estimates
\begin{align}\label{eq:gradient-estimate}
	\|\nabla u\|_{L^p(\Omega)}\leq d^2\frac{A_p^{1/p}}{K_{d,1}}\qquad\text{and} \qquad	 |u|_{BV(\Omega)}\leq d^2\frac{A_1}{K_{d,1}}.
\end{align}
\end{theorem}
\vspace{2mm}

 \noindent The constant $K_{d,1}$ appearing in \eqref{eq:gradient-estimate} is a universal constant independent of the geometry of $\Omega$ and is given by the following general mean value formula over the unit sphere 
\begin{align}\label{eq:BBM-constant}
K_{d,p}= \fint_{\mathbb{S}^{d-1}} |w\cdot e|^p \d \sigma_{d-1}(w)= \frac{\Gamma\big(\frac{d}{2}\big)\Gamma\big(\frac{p+1}{2}\big)}{\Gamma\big(\frac{d+p}{2}\big) \Gamma\big(\frac{1}{2}\big)}, 
\end{align}
\noindent for any unit vector $e\in \mathbb{S}^{d-1}$; see Proposition \ref{prop:cost-Kdp} for the computation. The constant $K_{d,p}$ also appears in \cite{BBM01}. There is a similar constant in \cite[Section 7]{IN10} when studying nonlocal approximations of the $p$-Laplacian. 
%The rotation invariance of the Hausdorff measure $\d\sigma_{d-1}$ of the $d-1$-dimensional sphere $\mathbb{S}^{d-1}$ implies that the constant $K_{d,p}$ is independent of the vector $e\in \mathbb{S}^{d-1}$. 
Observe that in general, for every $z\in \R^d$, we have 
\begin{align}\label{eq:rotation-invariant-constant}
\fint_{\mathbb{S}^{d-1}} |w\cdot z|^p \d \sigma_{d-1}(w) =|z|^p\fint_{\mathbb{S}^{d-1}} |w\cdot e|^p \d \sigma_{d-1}(w)=|z|^pK_{d,p}. 
\end{align}

\noindent Theorem \ref{thm:liminf} yields the following nonlocal characterization of constant functions; see also \cite{Brezis-const-function}. 
\begin{theorem}\label{thm:contant-funct}
	Assume $\Omega\subset \R^d$ is open and connected. If $u\in L^p(\Omega)$, $1\leq p<\infty,$ is such that $A_p=0$ then $u$ is almost everywhere constant on $\Omega$. 
\end{theorem}

Let us now comment about Theorem \ref{thm:liminf}. Observing that,  $\mathcal{E}^1_\Omega(u)\leq \mathcal{E}^i_\Omega(u), i=1,2,3$, Theorem \ref{thm:liminf} obviously remains true if the nonlocal forms of type $\mathcal{E}^1_\Omega$ are replaced with those of type $\mathcal{E}^2_\Omega$ or $\mathcal{E}^3_\Omega$. It is to be noted that, Theorem \ref{thm:liminf} is governed by two fundamental counter intuitive remarks. Firstly, the lack of reflexivity of $L^1(\Omega)$ implies that, in the case $p=1$, the function belongs to $BV(\Omega)$ and not necessarily in $W^{1,1}(\Omega)$. In other words, assuming $A_1<\infty$ is not enough to conclude that $u\in W^{1,1}(\Omega)$. We give here a mere counterexample in one dimension; see Counterexample \ref{Ex:counterexample-extension} for the general case.
 For  $d=1$ and $p\geq1$, we consider,
\begin{align}\label{eq:counter-example}
\text{$\Omega= (-1,0)\cup (0,1), \, u(x) =\frac12\mathds{1}_{[0,1)} (x) -\frac12\mathds{1}_{(-1,0)}(x)\, $ and $\, \nu_{\eps} (h) = 
	\frac{p\eps(1-\eps)}{2|h|^{1+(1-\eps)p}} .$}
\end{align}
For $p=1$, it is straightforwards to verify that $u\in BV(-1, 1)\setminus W^{1,1}(-1,1)$ whereas we find that $A_1=1$. The second remark indicates that, the converse of Theorem \ref{thm:liminf} is not necessarily true in general. Indeed, by adopting the above example in \eqref{eq:counter-example}, see also the counterexample \ref{Ex:counterexample-extension}, we find that $u\in W^{1,p}(\Omega)$ while $A_p=\infty$ for $p>1$. 
A reasonable explanation to the latter matter is that, $\Omega= (-1,0)\cup (0,1)$ is not an extension $W^{1,p}$-domain. To put it another way, this situation in particular (and in general) occurs due to the lack of the regularity of the boundary $\partial\Omega$. Therefore, to investigate the converse of Theorem \ref{thm:liminf}, we need some additional assumption on $\Omega$ such as the extension property. It is noteworthy to recall that $\Omega \subset \R^d$ is called to be a $W^{1,p}$-extension (resp. a $BV$-extension) domain if there exists a linear operator $E:W^{1,p}(\Omega)\to W^{1,p}(\R^d)$ (resp. $E: BV(\Omega)\to BV(\R^d)$) and a constant $C: = C(d,p,\Omega)$ depending only on the domain $\Omega$ and the dimension $d$ such that 
	\begin{align*}
	Eu\mid_{\Omega} &= u \qquad\hbox{and} \qquad \|Eu\|_{W^{1,p}(\R^d)}\leq C \|u\|_{W^{1,p}(\Omega)} \qquad\text{for all}\quad u \in W^{1,p}(\Omega)\\
	(\text{resp.}\quad
	Eu\mid_{\Omega}& = u \qquad\hbox{and} \qquad \|Eu\|_{BV(\R^d)}\leq C \|u\|_{BV(\Omega)} \qquad\quad\, \text{for all}\quad u \in BV(\Omega)).
	\end{align*}
	
\noindent Examples of extension domains include, bounded Lipschitz domains which are both $W^{1,p}$-extension and $BV$-extension domains. In particular euclidean balls and rectangles in $\R^d$ are extension domains. The upper half space $\R^d_+= \{(x',x_d)\in \R^d\,: x_d>0 \}$ is a simple example of an unbounded extension domain. The geometric characterization of extension domains has been extensively studied in the last decades. The $W^{1,p}$-extension property of an open set $\Omega$ infers certain regularity of the boundary $\partial\Omega$. For instance, according to \cite[Theorem 2]{HKT08}, a $W^{1,p}$-extension domain $\Omega\subset \R^d$ is necessarily is a $d$-set, i.e., satisfies the volume density condition, viz. there exists 
a constant $c>0$ such that $|\Omega\cap B(x,r)|\geq cr^d$  for all $x\in \partial \Omega$ and $0<r<1$. In virtue of the Lebesgue differentiation theorem, one finds that a $d$-set $\Omega$ is a Jordan set\cite{Zor16}, i.e., its boundary $\partial \Omega$ has Lebesgue measure zero, $|\partial\Omega|=0$. Subsequently, for a $W^{1,p}$-extension domain $\Omega$ there holds, 
\begin{align}\label{eq:lp-boundary-extension}
\text{$\int_{\partial \Omega} |\nabla Eu (x)|^p\d x=0$\quad for all $u\in W^{1,p}(\Omega)$.}
\end{align} 

\noindent To the best of our knowledge, the question whether the geometric characterization \eqref{eq:lp-boundary-extension} remains true for a $BV$-extension domain is still unknown. However, thanks to \cite[Lemma 2.4]{HKT08} or \cite[Theorem 1.3]{GR22} we know that every $W^{1,1}$-extension is a $BV$-extension domain. Throughout this article, we require a $BV$-extension domain $\Omega$ to satisfy the condition 
\begin{align}\label{eq:bv-boundary-extension}
\text{$|\nabla Eu | (\partial \Omega) =\int_{\R^d}\mathds{1}_{\partial \Omega}(x) \d|\nabla Eu |(x)=0$\quad for all $u\in BV(\Omega)$.}
\end{align}
\noindent It is to be noted that, in contrast to \eqref{eq:lp-boundary-extension}, having $|\partial \Omega|=0$ does not necessarily imply \eqref{eq:bv-boundary-extension}. Indeed, it suffices to consider once more the example \eqref{eq:counter-example} where one gets $\nabla u=\delta_0$ (the Dirac measure at the origin), so that $|\nabla u|(\partial \Omega)=1$. Some authors rather define a $BV$-extension domain together with the condition \eqref{eq:bv-boundary-extension}; see for instance \cite{AFP00, GR22}. Extended discussions on $BV$-extension domains can be found in \cite{KMS10,Lah15}. Several references on extension domains for Sobolev spaces can be found in \cite{Zh15}. Our second main result, which is an improved converse of Theorem \eqref{thm:liminf}, reads as follows. 
\begin{theorem}\label{thm:BBM-limit-result}
	Assume $\Omega\subset \R^d$ is a $W^{1,p}$-extension domain. If $u \in L^p(\Omega)$ with $1<p<\infty$ or $p=1$ and $u\in W^{1,1}(\Omega)$ then we have 
	\begin{align}\label{eq:w1p-BBM-limit}
	 \lim_{\varepsilon\to 0} \iil_{\Omega\Omega } |u(x)-u(y)|^p\nu_\varepsilon(x-y)\d y\d x = K_{d,p}\|\nabla u\|^p_{L^p(\Omega)} . 
	\end{align}
Moreover if $p=1$ and $\Omega$ is a $BV$-extension domain then for $u\in L^1(\Omega)$ we have 
\begin{align}\label{eq:bv-BBM-limit}
\lim_{\varepsilon\to 0} \iil_{\Omega\Omega } |u(x)-u(y)|\nu_\varepsilon(x-y)\d y\d x = K_{d,1}|u|_{BV(\Omega)}.
\end{align}

\end{theorem}

 \noindent We highlight that the counterexample \ref{Ex:counterexample-extension} shows that the conclusion of Theorem \ref{thm:BBM-limit-result} might be erroneous if $\Omega$ is not an extension domain. In one way of proving Theorem \ref{thm:BBM-limit-result}, we establish the following sharp version of the estimates in \eqref{eq:gradient-estimate} (see Theorem \ref{thm:liminf-BBM}) 
 \begin{align}\label{eq:gradient-estimate-sharp}
 \|\nabla u\|^p_{L^p(\Omega)}\leq \frac{A_p}{K_{d,p}}\qquad\text{and}\qquad | u|_{BV(\Omega)}\leq \frac{A_1}{K_{d,1}}. 
 \end{align}
 
 \noindent Indeed, Theorem \ref{thm:BBM-limit-result} shows that the estimates in \eqref{eq:gradient-estimate-sharp} turn into equalities provided that $\Omega$ is an extension domain. 
As immediate consequences of Theorem \ref{thm:liminf} and Theorem \ref{thm:BBM-limit-result} we have the following characterizations for the spaces $W^{1,p}(\Omega)$ and $BV(\Omega)$ when $\Omega$ is an extension domain. 

\begin{theorem}\label{thm:charact-w1p}
Assume $\Omega\subset \R^d$ is a $W^{1,p}$-extension domain, $p>1$ and let $u\in L^p(\Omega)$. Then $u\in W^{1,p}(\Omega)$ if and only if $A_p<\infty.$ Moreover, we have
		\begin{align*}
	&\lim_{\varepsilon\to 0}\iint\limits_{\Omega\Omega}|u(x)-u(y)|^p\nu_\varepsilon(x-y) \d y\d x = K_{d,p} \|\nabla u\|^p_{L^p(\Omega)}. 
	\end{align*}
\end{theorem}

%\smallskip 
\begin{customthm}{\ref{thm:charact-w1p}'}\label{thm:charact-BV}
\textit{ Assume $\Omega\subset \R^d$ is a $BV$-extension domain, $p=1$, and let $u\in L^1(\Omega)$. Then $u\in BV(\Omega)$ if and only if $A_1<\infty$. Moreover, we have 
	\begin{align*}
	&\lim_{\varepsilon\to 0}\iint\limits_{\Omega\Omega}|u(x)-u(y)|\nu_\varepsilon(x-y) \mathrm{d}y\mathrm{d}x = K_{d,1} |u|_{BV(\Omega)}. 
	\end{align*} }
\end{customthm}

\noindent In contrast to the forms of type $\mathcal{E}^1_\Omega$, the collapse phenomenon across $\partial\Omega$ occurs for the forms of type $\mathcal{E}^2_\Omega$ or $\mathcal{E}^3_\Omega$ in Theorem \ref{thm:BBM-limit-result}. 

%\noindent The collapse phenomenon across $\partial\Omega$ occurs if we usethe forms of type $\mathcal{E}^2_\Omega$ or $\mathcal{E}^3_\Omega$ in Theorem \ref{thm:BBM-limit-result}. 
 
\begin{theorem}\label{thm:collapsing-accross-boundary}
Assume  that $(i)$ $\Omega\subset \R^d$ is a $W^{1,p}$-extension domain 
or that $(ii)$ $\R^d\setminus\overline{\Omega}$ is a $W^{1,p}$-extension domain  and $\partial\Omega= \partial \overline{\Omega}$. For $u\in W^{1,p}(\R^d)$, $p\geq 1$, there hold the following limits:
%or $u\in BV(\R^d)$ for $p=1$. Let us  we abuse the notation $\|\nabla u\|_{L^1(\R^d)} =|u|_{BV(\R^d)}$. 
\begin{align*}
&\lim_{\eps\to 0}\iil_{ \Omega \Omega^c} |u(x)-u(y)|^p\nu_\eps(x-y)\d y\d x= 0,\\ %\label{eq:form-conv-accross}\\
&\lim_{\eps\to 0}\iil_{ \Omega \Omega} |u(x)-u(y)|^p\nu_\eps(x-y)\d y\d x= K_{d,p}\|\nabla u\|^p_{L^p(\Omega)},\\
&\lim_{\eps\to 0}\iil_{ \Omega \R^d} |u(x)-u(y)|^p\nu_\eps(x-y)\d y\d x= K_{d,p}\|\nabla u\|^p_{L^p(\Omega)},\\ 
& \lim_{\eps\to 0} \hspace{-1ex}\iil_{ (\Omega^c \times\Omega^c)^c} \hspace{-2ex}|u(x)-u(y)|^p\nu_\eps(x-y)\d y\d x= K_{d,p}\|\nabla u\|^p_{L^p(\Omega)}. 
	\end{align*}
Moreover, for $u\in BV(\R^d)$ and  $p=1$, the above limits remain true provided that $|\nabla u|(\partial\Omega)=0$.
% $\int_{\partial\Omega}\hspace{-0.5ex}\d |\nabla u|=0$. 
\end{theorem}
%\smallskip 

\begin{proof} %as  the boundary of an extension domains is a null set.  
In fact, in both cases $(i)$ and $(ii)$ we have  $|\partial\Omega|=|\partial\overline{\Omega}|=0$. Thus, Theorem \ref{thm:collap-bdary} yields the first limit.  For the case $(ii)$, the remaining limits follow from Theorem \ref{thm:BBM-limit-result} since $\Omega\times \Omega=(\R^d\times \R^d)\setminus[\Omega^c\times \Omega\cup\Omega\times \Omega^c\cup\Omega^c\times \Omega^c]$, $\Omega\times \R^d= \Omega\times \Omega\cup \Omega\times \Omega^c$ and   $(\Omega^c\times \Omega^c)^c =(\R^d\times \R^d)\setminus(\Omega^c\times \Omega^c)$. The case $(i)$ is obtain by interchanging $\Omega$ and $\Omega^c$. The  situation $u\in BV(\R^d)$ is analogous. 
%
%Since we assume the condition \eqref{eq:bv-boundary-extension} for BV-extension domains, 
%The remaining limits follow from  the situation where $\R^d$, $\Omega$ and/or $\R^d\setminus\overline{\Omega}$ are $W^{1,p}$-extension domains. Indeed, the case $(i)$ follows since $ (\Omega^c\times \Omega^c)^c =\Omega\times\Omega\cup \Omega\times \Omega^c\cup \Omega^c\times \Omega $  and $\Omega\times \R^d= \Omega\times \Omega\cup \Omega\times \Omega^c$.
%The case $(ii)$ is analogous since   $\Omega\times \Omega=(\R^d\times \R^d)\setminus[\Omega^c\times \Omega\cup\Omega\times \Omega^c\cup\Omega^c\times \Omega^c]$ and   $(\Omega^c\times \Omega^c)^c =(\R^d\times \R^d)\setminus(\Omega^c\times \Omega^c)$.
%
%\noindent To derive Theorem \ref{thm:collapsing-accross-boundary}, it suffices to apply Theorem \ref{thm:BBM-limit-result} for the domains $\R^d,~\Omega$ and $ \Omega^c$, which are of course extension domains, by observing that $\Omega\times \Omega^c\cup \Omega^c\times \Omega= (\R^d\times \R^d)\setminus[\Omega\times \Omega \cup \Omega^c\times \Omega^c]$, $\Omega\times \R^d= \Omega\times \Omega\cup \Omega\times \Omega^c$ and $ (\Omega^c\times \Omega^c)^c = (\R^d\times \R^d)\setminus(\Omega^c\times \Omega^c)$.
\end{proof}
%
%\smallskip

\noindent  %Note that the limits in Theorem \ref{thm:collapsing-accross-boundary} remain true when $\Omega$ replaced with $\Omega^c$. 
The next result, is an alternative to Theorem \ref{thm:BBM-limit-result} if $\Omega$ is not an extension domain.

\begin{theorem}\label{thm:weak-convergence}
Let $\Omega\subset \R^d$ be open. Let $u\in W^{1,p}(\Omega)$ and define the Radon measures 
	\begin{align*}
\d \mu_{\varepsilon}(x) = \int_{\Omega}|u(x)-u(y)|^p\nu_\varepsilon(x-y) \mathrm{d}y\d x.
\end{align*}
The sequence $(\mu_{\varepsilon})_\eps$
converges weakly on $\Omega$ (in the sense of Radon measures) to the Radon measure $\d \mu(x) =K_{d,p}|\nabla u(x)|^p\d x$, 
 i.e. $ \mu_\varepsilon (E)\xrightarrow{\eps\to 0}\mu(E)$ for every compact set $E\subset \Omega$. Moreover, if $p=1$ and $u\in BV(\Omega)$ then $\d \mu(x) =K_{d,1}\d |\nabla u|(x)$. 
\end{theorem}

\smallskip

\noindent Note in particular that, Theorem \ref{thm:weak-convergence} implies that $(\mu_\eps)_\eps$ \emph{}vaguely convergence to $\mu$, i.e., 
\begin{align*}
\lim_{\eps \to 0}\int_\Omega \varphi(x)\d \mu_\eps(x) =\int_\Omega \varphi(x)\d \mu(x),\quad\text{for every $\varphi\in C_c^\infty(\Omega)$}.
\end{align*}

 % viz., they proved that a function $u\in L^p(\Omega)$ lies in $W^{1,p}(\Omega)$ for $1<p<\infty$ or in $ BV(\Omega)$ for $p=1$ provided that 
 %\begin{align}\label{eq:BBM-liminf-condition}
 %\liminf_{\varepsilon\to 0} \iil_{\Omega\Omega } \frac{|u(x)-u(y)|^p}{|x-y|^p}\rho_\varepsilon(x-y)\d x\d y <\infty\,. 
 %\end{align}
 %\begin{align}\label{eq:BBM-limit}
 %\lim_{\varepsilon\to 0}\iil_{\Omega\Omega } \frac{|u(x)-u(y)|^p}{|x-y|^p}\rho_\varepsilon(x-y)\d x\d y = K_{d,p}\|\nabla u\|^p_{L^p(\Omega)}. 
 %\end{align} 
 
Let us comment on Theorem \ref{thm:liminf} and Theorem \ref{thm:BBM-limit-result}  and related results in the literature. Bourgain-Brezis-Mironescu \cite[Theorem 3' \& Theorem 2]{BBM01} proved the characterization Theorem \ref{thm:liminf}  under the stronger condition that $\Omega\subset \R^d$ is bounded Lipschitz and while considering the sub-class $\nu_\eps(h) =c_\eps|h|^{-p}\rho_\eps (h)$ where $(\rho_\eps)_\eps$ satisfies \eqref{eq:approx-dirac}. Beside this, with the same assumptions, Bourgain-Brezis-Mironescu in \cite[Theorem 2]{BBM01} also established the relation \eqref{eq:w1p-BBM-limit}. The case $\Omega=\R^d$ is also investigated by Brezis \cite{Brezis-const-function} while characterizing constant functions.
The case $p=1$, i.e., the relation \eqref{eq:bv-BBM-limit}, is also a natural subject of discussions in \cite{BBM01} wherein, the authors succeeded in the one dimensional setting when $\Omega=(0,1)$, viz., they proved that 
\begin{align*}
\int_0^1\int_0^1 \frac{|u(x)-u(y)|}{|x-y|}\rho_\eps(x-y)\d y\d x= K_{1,1}|u|_{BV(0,1)}\quad\text{for all } u\in BV(0,1).
\end{align*}
The general case $d\geq 2$ was completed later in \cite{Dav02} when $\Omega$ is a bounded Lipschitz domain. In this perspective, \cite[Lemma 2]{Dav02} also established a variant of Theorem \ref{thm:weak-convergence} for the case $p=1$. Clearly, our setting of Theorem \ref{thm:liminf} is more general as no restriction on $\Omega$ is required and, in the sense mentioned above, the class $(\nu_\eps)_\eps$ satisfying \eqref{eq:plevy-approx} is strictly lager than that of $(\rho_\eps)_\eps$ satisfying \eqref{eq:approx-dirac}. In addition, in contrast to \cite{BBM01}, $\Omega$ is not necessarily bounded in Theorem \ref{thm:BBM-limit-result} and that the situation where $\Omega$ has a Lipschitz boundary appears as a particular case of Theorem \ref{thm:BBM-limit-result}.
We point out that Theorem \ref{thm:BBM-limit-result} is reminiscent of \cite[Theorem 3.4]{FGKV19} for $p=2$. Ultimately, let us mention that, after the release of the first version of this work, the authors of \cite{BMR20} brought to our attention that they also established the relation \eqref{eq:w1p-BBM-limit} when $\nu_\eps(h) =\eps |h|^{-d-(1-\eps)p}$ (fractional kernels) for $1<p<\infty$. The case $p=1$ is, however, not fully covered therein. Our approach in this paper extends the works from \cite{BBM01, Brezis-const-function,Dav02,Pon04}. In the wake of \cite{BBM01}, several works regarding the characterization of Sobolev spaces and alike spaces have appeared in literature in the recent years. For example \cite{Pon04-gamma,Lud14} for characterization of Sobolev spaces via families of anisotropic interacting kernels, \cite{PS17,LS11} for characterization of BV spaces, \cite{MS02} for a study of asymptotic sharp fractional Sobolev inequality, \cite{Bra18} for characterization of Besov type spaces of higher order and \cite{FGKV19} for the study of Mosco convergence of nonlocal quadratic forms. 

This article is organized as follows. In the second section we address some examples of approximating sequence $(\nu_\eps)_\eps$ and some nonlocal spaces in connection with function of type $\nu_\eps$. The third section is devoted to the proofs of Theorem \ref{thm:liminf}, Theorem \ref{thm:BBM-limit-result} and Theorem \ref{thm:weak-convergence}. 

Throughout this article, $\eps>0$ is a small quantity tending to $0$. We frequently use the convex inequality $(a+b)^p\leq 2^{p-1}(a^p+b^p)$ for $a>0, b>0$, the Euclidean scalar product of $x=(x_1,x_2,\cdots, x_d)\in \R^d $ and $ y=(y_1,y_2,\cdots, y_d)\in \R^d$ is $x\cdot y =x_1x_1+x_2y_2+\cdots+ x_dy_d$ and denote the norm of $x$ by $|x| =\sqrt{x\cdot x}$. The conjugate of $p\in [1,\infty)$ is denoted by $p'$, i.e., $p+p'=pp'$ with the convention $1'=\infty$. Throughout, $|\mathbb{S}^{d-1}|$ denotes the area of the $d-1$-dimensional unit sphere, where we adopt the convention that $|\mathbb{S}^{d-1}|=2$ if $d=1$.

%%%%%%%%%%%%%%%%%%%%%%%%%%%%%%%%%%%%%%%%%%%%%%%%%%%%%%%%%%%%%%%%%%%%%%%
\section{Preliminaries}\label{sec:approx-dirac-mass}
%%%%%%%%%%%%%%%%%%%%%%%%%%%%%%%%%%%%%%%%%%%%%%%%%%%%%%%%%%%%%%%%%%%%%%%

\subsection{$p$-L\'{e}vy integrability and approximation of Dirac measure}\label{sec:approx-dirac-prop}
%\noindent Let us coin the concept of $p$-L\'{e}vy measure. 
\begin{definition} \label{def:plevy-approx} %Let $1\leq p<\infty$.
\begin{enumerate}[$(i)$]
		\item 
	A nonnegative Borel measure $\nu(\d h)$ on $\R^d$ is called a $p$-L\'{e}vy measure if $\nu(\{0\})=0$ and it satisfies the $p$-L\'{e}vy integrability condition; that is to say 
	\begin{align*}
	\int_{\R^d} (1\land |h|^p)\nu(\d h)<\infty. 
	\end{align*}
\item 
		A family $(\nu_\varepsilon)_\varepsilon$ satisfying \eqref{eq:plevy-approx} is called \emph{a Dirac approximation of $p$-L\'{e}vy measures.}
	\end{enumerate}
	\end{definition}

\smallskip 

\noindent 	Patently, one recovers the usual definition of L\'{e}vy measures when $p=2$. Such measures are paramount in the study of stochastic process of L\'{e}vy type; see for instance  \cite{Sat13,App09,Ber96} for further details.  We intentionally omit the dependence of $\nu$ and $\nu_\eps$ on $p$. This dependence will be always clear from the context. The following result shows that by rescaling appropriately a radial $p$-L\'{e}vy integrable function $\nu(h)$ one obtains a family $(\nu_\varepsilon)_\varepsilon$ satisfying \eqref{eq:plevy-approx}.

\begin{proposition}\label{prop:rescalled}
%Assume $\nu : \R^d\setminus\{0\} \to [0, \infty)$ is a positive measurable function satisfying the $p$-L\'{e}vy integrability condition, i.e., 
Let  $\nu \in L^1( \R^d,1\land |h|^p)$ with $\nu\geq0$. Define the rescaled family $(\nu_\varepsilon)_\varepsilon$, as
	\begin{align}\label{eq:rescalled-levy-measure}
	\begin{split}
	\nu_\varepsilon(h) = 
	\begin{cases}
	\varepsilon^{-d-p}\nu\big(h/\varepsilon\big)& \text{if}~~|h|\leq \varepsilon,\\
	\varepsilon^{-d}|h|^{-p}\nu\big(h/\varepsilon\big)& \text{if}~~\varepsilon<|h|\leq 1,\\
	\varepsilon^{-d}\nu\big(h/\varepsilon\big)& \text{if}~~|h|>1.
	\end{cases}
	\end{split}
	\end{align}
\vspace{-1mm}
\noindent Then for every $\delta>0, \eps\in (0,1)$
	\begin{align*}
	&\int_{\R^d}(1\land |h|^p)\nu_\varepsilon(h)\d h= \int_{\R^d}(1\land |h|^p)\nu(h)\d h
	\quad \text{and}\quad \lim_{\varepsilon\to 0 } \int_{|h|>\delta }(1\land |h|^p)\nu_\varepsilon(h)\d h=0 \, .
	\end{align*}
\end{proposition}
\begin{proof}
Since $\nu \in L^1(\R^d, 1\land |h|^p)$ the dominated convergence theorem yields
	\begin{align*}
	\lim_{\varepsilon\to 0 } \int_{|h|>\delta }(1\land |h|^p)\nu_\varepsilon(h)\d h=\lim_{\varepsilon\to 0 } \int_{|h|>\delta/\varepsilon }(1\land |h|^p)\nu(h)\d h=0.
	\end{align*}
	We omit the remaining details as it solely involves straightforward computations.
\end{proof}

\smallskip

\noindent The behavior of the rescaled family $(\nu_\varepsilon)_\varepsilon$ in \eqref{eq:rescalled-levy-measure},  $p=2$, is governed by two keys observations. 
The first is that it gives raise to a family of L\'{e}vy measures with a concentration property. Secondly, from a probabilistic point of view, one obtains a family of pure jumps L\'{e}vy processes $(X_\varepsilon)_\varepsilon$ each associated with the measure $\nu_\varepsilon(h)\mathrm{d} h$ from a L\'{e}vy process $X$ associated with $\nu(h)\mathrm{d} h$. In fact, the family of stochastic processes $(X_\varepsilon)_\varepsilon$ converges in finite dimensional distributional sense (see \cite{FGKV19}) to a Brownian motion provided that one in addition assumes that $\nu$ is radial. Proposition \ref{prop:dirac-approx} $(ii)$ below shows that the generator of the process $X_\eps$ denoted $L_\eps$ (see \eqref{eq:integro-diff}), converges to $-\frac{1}{2d}\Delta$ which is the generator of a Brownian motion. In short, rescaling via \eqref{eq:rescalled-levy-measure} any isotropic pure jump L\'{e}vy process leads to a Brownian motion. This, could be one more argument to back up the ubiquity of the Brownian motion. The convergence highlighted above is involved in a more significant context. For example in \cite{FGKV19}, the convergence in Mosco sense of the Dirichlet forms associated with process in play is established. Beside these observations, the works \cite{Fog20, FK22} establish that if $\Omega$ is bounded with a Lipschitz boundary and $u_\eps$ satisfies in the weak sense nonlocal problems of the $L_\eps u_\eps= f$ in $\Omega$ augmented with Dirichlet condition $u_\eps=0$ on $\Omega^c$ (resp. Neumann condition $\mathcal{N}u_\eps=0$ on $\Omega^c$) condition then $(u_\eps)_\eps$ converges in $L^2(\Omega)$ to some $u\in W^{1,2}(\Omega)$, where $u$ is the weak solution  to the local problem $-\frac{1}{2d}\Delta u= f$ in $\Omega$ augmented with Dirichlet boundary $u=0$ on $\partial\Omega$ (resp. Neumann condition $\frac{\partial u}{\partial n} =0$ on $\partial \Omega$). Here,  $L_\eps$ is given by \eqref{eq:integro-diff} and $\mathcal{N}_\eps$ is defined by
 \begin{align*}
\mathcal{N}_\eps u(x):= \int_{\Omega} (u(x)-u(y))\nu_\eps(x-y)\d y. 
\end{align*}

%\smallskip

\begin{remark}\label{rem:asymp-nu}
Assume the family $(\nu_\eps)_\eps$ satisfies \eqref{eq:plevy-approx}.
%\begin{enumerate}[$(i)$]
%	\item 
	Let $\beta\in \R,$ then for all $R>0$ we have
\begin{align*}
\lim_{\varepsilon\to 0}\int_{|h|\leq R} (1\land |h|^\beta)\nu_\varepsilon(h)\mathrm{d}h =
\begin{cases}
0\quad \text{ if}\quad \beta >p\\
1\quad \text{ if}\quad \beta =p\\
\infty \quad \text{if}\quad \beta <p.
\end{cases}
\end{align*}
Indeed, for fixed $\delta>0$, \eqref{eq:plevy-approx} implies
\begin{align*}
&\lim_{\varepsilon \to 0} \il_{ \delta<|h|\leq R} (1\land |h|^p) \nu_\varepsilon(h) \mathrm{d}h
\leq \lim_{\varepsilon \to 0} \il_{ |h|>\delta} (1\land |h|^p) \nu_\varepsilon(h) \mathrm{d}h=0,
\\
&\lim_{\varepsilon \to 0}\int_{ |h|\leq \delta} (1\land |h|^p) \nu_\varepsilon(h) \mathrm{d}h
=1- \lim_{\varepsilon \to 0} \int_{ |h|> \delta} (1\land |h|^p) \nu_\varepsilon(h) \mathrm{d}h=1. 
\end{align*}
	Thus for $\beta>p$ we have 
	\begin{align*}
	\lim_{\varepsilon \to 0}\il_{|h|\leq R} \hspace{-1ex}(1\land |h|^\beta)\nu_\varepsilon(h)\mathrm{d}h
	&\leq \lim_{\varepsilon \to 0} \Big( R^{\beta-p} \hspace{-3ex}\il_{\delta< |h|\leq R} \hspace{-2ex}(1\land |h|^p)\nu_\varepsilon(h) \mathrm{d}h + \delta^{\beta-p} \hspace{-2ex}\il_{|h|\leq\delta}\hspace{-1ex} (1\land |h|^p)\nu_\varepsilon(h) \mathrm{d}h\Big) = \delta^{\beta-p}.
	\end{align*}
	Likewise for $\beta<p$ we have 
	\begin{align*}
	\lim_{\varepsilon \to 0}\il_{|h|\leq R} \hspace{-1ex}(1\land |h|^\beta)\nu_\varepsilon(h)\mathrm{d}h
	&\geq \lim_{\varepsilon \to 0} \Big( R^{\beta-p} \hspace{-3ex}\il_{\delta< |h|\leq R} \hspace{-2ex}(1\land |h|^p)\nu_\varepsilon(h) \mathrm{d}h + \delta^{\beta-p} \hspace{-2ex}\il_{|h|\leq\delta}\hspace{-1ex} (1\land |h|^p)\nu_\varepsilon(h) \mathrm{d}h\Big) = \delta^{\beta-p}.
	\end{align*}
In either case, letting $\delta \to 0$ provides the claim.
\end{remark}

\begin{remark}\label{rem:asymp-nu-bis}Assume the family $(\nu_\eps)_\eps$ satisfies \eqref{eq:plevy-approx}. Note that the relation 
	\begin{align}\label{eq:concentration-prop}
\lim_{\varepsilon\to 0}\int_{|h|> \delta }(1\land |h|^p)\nu_\varepsilon(h)\d h=0,
	\end{align}
is often known as the concentration property and is merely equivalent to 
	\begin{align*}
	\lim_{\varepsilon\to 0}\int_{|h|>\delta }\nu_\varepsilon(h)\,\d h=0,\quad\text{for all}\quad \delta >0.
	\end{align*}
\noindent Indeed, for all $\delta >0$ we have
\begin{align*}
\int_{|h|>\delta }(1\land |h|^p)\nu_\varepsilon(h)\,\d h\leq \int_{|h|> \delta }\nu_\varepsilon(h)\,\d h\leq (1\land \delta^p)^{-1}\hspace{-1ex}\int_{|h|> \delta }(1\land |h|^p)\nu_\varepsilon(h)\,\d h.
\end{align*}

Consequently, for all $\delta >0$ we also have 
	\begin{align}\label{eq:concentration-bis}
	\lim_{\varepsilon\to 0}\int_{|h|\leq \delta }(1\land |h|^p)\nu_\varepsilon(h)\,\d h= \lim_{\varepsilon\to 0}\int_{|h|\leq \delta }|h|^p\nu_\varepsilon(h)\,\d h=1.\quad
	\end{align}
%\end{enumerate} 
\end{remark}

\smallskip

\noindent The next result infers certain some convergences  of the family $(\nu_\varepsilon)_\varepsilon$ for the case $p=1$ and $p=2$. 
% thus somehow  withstands Definition \ref{def:plevy-approx} $(ii)$.
%
\begin{proposition}\label{prop:dirac-approx} Consider the family $(\nu_\eps)_\eps$ satisfying \eqref{eq:plevy-approx}.
	
\begin{enumerate}[$(i)$]
\item If $p=1$ then we have $\langle\nu_\varepsilon, \varphi-\varphi(0) \rangle\xrightarrow{\varepsilon\to 0}0$  for every $\varphi\in C_c^\infty(\R^d)$.
%Here $\delta_0$ stands for the Dirac mass at the origin.

\item If $p=2$ then for a function $u: \R^d\to \mathbb{R}$ which is bounded and $C^2$ on a neighborhood of $x$, 
 $$\lim_{\eps \to 0} L_\eps u(x) =-\frac{1}{2d}\Delta u(x),$$
 where $\Delta$ is the Laplace operator and $L_\eps$ is the integrodifferential operator
 \begin{align}\label{eq:integro-diff}
 L_\eps u(x) := -\frac{1}{2}\int_{\R^d }(u(x+h)+ u(x-h) -2u(x)) \nu_{\eps}(h)\,\d h.
 \end{align}
\end{enumerate}
\end{proposition}

\begin{proof}
	$(i)$ Let $\varphi\in C_c^\infty(\R^d)$ using the fundamental theorem of calculus we can write 
	\begin{align*}
	\langle\nu_\varepsilon, \varphi-\varphi(0)\rangle &=\int_{\R^d}(\varphi(h)-\varphi(0))\, \nu_\varepsilon(h)\d h \\
	&=  \int_{|h|\leq 1}(\varphi(h)-\varphi(0)-\nabla \varphi(0)\cdot h)\, \nu_\varepsilon(h) \d h + \int_{|h|\geq 1}(\varphi(h)- \varphi(0))\nu_\varepsilon(h)\d h\\
	&=\int_{|h|\leq 1}\int_0^1\int_0^1 s((D^2 \varphi(tsh)\cdot h)\cdot h)\, \nu_\varepsilon(h) \d s\d t\d h + \int_{|h|\geq 1}(\varphi(h)- \varphi(0))\nu_\varepsilon(h)\d h .
	\end{align*}
	The conclusion clearly follows since 
	\begin{align*}
	\Big| \int_{|h|>1}(\varphi(h)- \varphi(0))\nu_\varepsilon(h)\d h\Big|\leq 2\|\varphi\|_{L^\infty(\R^d)} \int_{|h|> 1}\nu_\varepsilon(h)\d h\xrightarrow{\varepsilon\to 0}0
	\end{align*}
	and by Remark \ref{rem:asymp-nu} we have 
	\begin{align*}
	\Big|\int_{|h|\leq 1}\int_0^1\int_0^1 s((D^2 \varphi(tsh)\cdot h)\cdot h)\, \nu_\varepsilon(h) \d s\d t\, \d h\Big|\leq \|D^2\varphi\|_{L^\infty(\R^d)} \int_{|h|\leq 1}|h|^2\nu_\varepsilon(h)\d h\xrightarrow{\varepsilon\to 0}0.
	\end{align*}

\noindent $(ii)$ Note that $D^2u$ is bounded in a neighborhood of $x$. Hence, for $0<\delta<1$ sufficiently small, for all $|h|<\delta$ we have the estimate
		\begin{align*}
		\left|u(x+h)+u(x-h)-2u(x)\right|\leq 4(\|u\|_{ C_b(\R^d)}+\|D^2u\|_{ C(B_{4\delta}(x) })(1 \land |h|^2). 
		\end{align*}
The boundedness of $u$ implies that 
		\begin{align*}
		\lim_{\eps\to 0}\int_{|h|>\delta} \left|u(x+h)+u(x-h)-2 u(x)\right|\nu_\eps(h)\d h =4\|u\|_{L^\infty(\R^d)} \lim_{\eps\to 0}\int_{|h|>\delta} \nu_\eps(h)\d h=0.
		\end{align*}

\noindent Since the Hessian of $u$ is continuous at $x$, given $\eta>0$ we have $|D^2(x+z) -D^2 u(x)|<\eta $ for $|z|<4\delta$ with $\delta>0$ sufficiently small, Remark \ref{rem:asymp-nu} implies
	\begin{align*}
&\lim_{\eps\to 0}\frac{1}{2} \int_{0}^{1} \int_{0}^{1} 2t \int_{|h|\leq \delta} |[(D^2 u(x-th + 2sth)-D^2u(x) ) \cdot h]\cdot h |\nu_\eps(h)\d h \, \d s\d t\\
&\leq \frac{\eta}{2} \lim_{\eps\to 0} \int_{|h|\leq \delta}(1\land |h|^2)\nu_\eps(h)\d h
=\frac{\eta}{2}.
\end{align*} 
Thus, the leftmost expression vanishes since $\eta>0$ is arbitrarily. Next, by symmetry we have  $\int_{|h|\leq\delta} h_i h_j\nu_\eps(h)\d h=0$ for $i\neq j$. The rotation invariance of the Lebesgue measure implies

\begin{align*}
\int_{|h|\leq\delta} [D^2u(x)\cdot h] \cdot h ~\nu_\eps(h)\d h
& = \sum_{\stackrel{i\neq j}{i, j=1}}^{d} 
\int_{|h|\leq\delta} \hspace{-1ex}\partial_{ij}^2u(x) h_i h_j \nu_\eps(h)\d h+ 
\sum_{i=1}^{d} \partial_{ii}^2u(x) \int_{|h|\leq\delta} \hspace{-1ex}h_i^2 \nu_\eps(h)\d h\\
&= \Delta u(x) \int_{|h|\leq \delta}\hspace{-1ex}  h_1^2\nu_\eps(h)\d h
= \frac{1}{d}\Delta u(x) \int_{|h|< \delta} \hspace{-1ex} |h|^2\nu_\eps(h)\d h\\
&= \frac{1}{d} \Delta u(x)\int_{|h|\leq\delta} \hspace{-1ex} (1\land |h|^2)~\nu_\eps(h)\d h
\xrightarrow{\eps\to 0} \frac{1}{d} \Delta u(x). 
\end{align*}

\noindent Finally, by the fundamental theorem of calculus we find that
\begin{align*}
& -\frac{1}{2} \int_{|h|\leq \delta} \left(u(x+h)+u(x-h)-2(x)\right)\nu_\eps(h)\d h 
= -\frac{1}{2}\int_{|h|\leq \delta} [D^2u(x)\cdot h] \cdot h ~\nu_\eps(h)\d h 
\\ &-\frac{1}{2} \int_{0}^{1} \int_{0}^{1} 2t \int_{|h|\leq \delta} [D^2 u(x-t h + 2 st h)\cdot h-D^2 u(x) \cdot h] \cdot h ~\nu_\eps(h)\d h \, \d s\d t\xrightarrow{\eps\to 0} -\frac{1}{2d} \Delta u(x).
\end{align*}
\end{proof}

\noindent Let us give examples of $\nu_\eps$ satisfying \eqref{eq:plevy-approx}. 
The first  example is related to fractional Sobolev spaces.
 
\begin{example}\label{Ex: stable-class}
 The family $(\nu_\varepsilon)_\varepsilon$ of kernels defined for $h \neq 0$ by
	$$\nu_\varepsilon(h) = a_{\varepsilon, d,p} |h|^{-d-(1-\varepsilon)p}\quad\text{with }\quad a_{\varepsilon, d,p} = \frac{p\varepsilon(1-\varepsilon)}{|\mathbb{S}^{d-1}|}.$$ 
	
\end{example}

\noindent The next class of examples is that of Proposition \ref{prop:rescalled}.
\begin{example}
	Assume $\nu: \R^d \setminus\{0\}\to [0, \infty]$ is radial and consider the family $(\nu_\varepsilon)_\varepsilon$ such that each $\nu_\varepsilon$ is the rescaling of $\nu$ defined as in \eqref{eq:rescalled-levy-measure} provided that 
	\begin{align*}
	\int_{\R^d}(1\land |h|^p)\nu(h)\d h=1. 
	\end{align*}
A subclass is obtained if one considers an integrable radial function $\rho:\R^d \to [0,\infty]$ and defines $\nu(h) = c|h|^{-p} \rho(h)$ for a suitable normalizing constant $c>0$. 
\end{example}

%\smallskip
 
\begin{example}\label{Ex:rho-var}
	Assume $(\rho_\varepsilon)_\varepsilon$ is an approximation of the unity, i.e., satisfies \eqref{eq:approx-dirac}. For instance, define $\rho_\varepsilon (h) =\varepsilon^{-d} \rho(h/\varepsilon)$ where $\rho\geq 0$ is radial and $\int_{\R^d}\rho(h)\d h=1$. 
	Define the family $(\nu_\varepsilon)_\varepsilon$ by $\nu_\varepsilon(h) = c_\eps |h|^{-p} \rho_\varepsilon(h),$ where $c_\eps>0$ is a normalizing constant given by 
	\begin{align*}
	c_\eps^{-1}= \int_{|h|\leq 1} \rho_\eps(h)\d h+ \int_{|h|> 1} |h|^{-p}\rho_\eps(h)\d h, 
	\end{align*}
	 for which the $p$-L\'{e}vy integrability condition in \eqref{eq:plevy-approx} holds. Note that $c_\eps \to 1$ as $\eps\to0.$
\end{example}

\smallskip

%\noindent Let us collect some concrete examples of sequence $(\nu_\varepsilon)_\varepsilon$.
\begin{example} \label{Ex:example-poincre1} Let $0<\varepsilon < 1$ and $\beta >-d$. We define 
	\begin{align*}
	\nu_\varepsilon(h) = \frac{d+\beta}{ |\mathbb{S}^{d-1}|\varepsilon^{d+\beta}} |h|^{\beta-p}\mathds{1}_{B_{\varepsilon}}(h). 
	\end{align*} 
	
\noindent Some special cases are obtained for $ \beta \in \{0, \eps p-d,p\}$. 
For the limiting case $\beta=-d$, we  put
	\begin{align*}
	\nu_\varepsilon(h) = \frac{1}{|\mathbb{S}^{d-1}|\log(\varepsilon_0/\varepsilon )}|h|^{-d-p}\mathds{1}_{ B_{\varepsilon_0}\setminus B_\varepsilon}(h). 
	\end{align*} 

\end{example}

\begin{example} \label{item-example-poincre2}
	%%%%%%%%%%%%%% Important comment %%%%%%%%%%%%
	
% In all the examples Below to recover the constant $b_\varepsilon$ just apply successively the change of variables: Polar coordinates, $r= \varepsilon u$ and $t = \frac{1}1+u}$. To find the asymptotic behaviors just apply binomial formula to the term in $(1-t)\sim 1 $
Let $0<\varepsilon < \varepsilon_0<1$ and $\beta >-d$. Define
\begin{align*}
\nu_\varepsilon(h) = \frac{(|h|+\varepsilon)^{\beta}|h|^{-p}}{|\mathbb{S}^{d-1}| b_\varepsilon}\mathds{1}_{B_{\varepsilon_0}}(h)\qquad \hbox{with }\quad b_\varepsilon = \varepsilon^{d+\beta}\int^{1}_{\frac{\varepsilon}{\varepsilon+\varepsilon_0}} t^{-d-\beta-1}(1-t)^{d-1}\d t. 
\end{align*} 
For the limiting case $\beta=-d$ consider

\begin{align*}
\nu_\varepsilon(h) = \frac{(|h|+\varepsilon)^{-d}|h|^{-p}}{|\mathbb{S}^{d-1}| |\log\eps| b_\varepsilon}\mathds{1}_{B_{\varepsilon_0}(h)}\qquad \hbox{with }\quad b_\varepsilon = |\log\eps|^{-1}\int^{1}_{\frac{\varepsilon}{\varepsilon+\varepsilon_0}} t^{-1}(1-t)^{d-1}\d t. 
\end{align*}
	
\noindent In either case the constant $b_\varepsilon \to 1$ as $\eps\to 0$ and is such that $\int_{\R^d}(1\land |h|^p) \nu_\varepsilon(h) \d h =1$. 
Another example familiar to the case $\beta=-d$ is 
\begin{align*}
\nu_\varepsilon(h) = \frac{(|h|+\varepsilon)^{-d-p}}{|\mathbb{S}^{d-1}| |\log\eps| b_\varepsilon}\mathds{1}_{B_\varepsilon(h)}\qquad \hbox{ with }\quad b_\varepsilon = |\log\eps|^{-1} \int^{1}_{\frac{\varepsilon}{\varepsilon+\varepsilon_0}} t^{-1}(1-t)^{d+p-1}\d t. 
\end{align*} 
\end{example}

\subsection{Local and nonlocal spaces}
Let $\nu: \R^d\setminus\{0\}\to [0,\infty)$ be $p$-L\'{e}vy integrable and $\mathcal{E}^i_\Omega(u), \, i=1,2,3$ be the forms defined in \eqref{eq:nonlocal-form}. The space $\WnuOm= \big\{u \in L^p(\Omega)~:|u|^p_{\WnuOm}<\infty \big \} $ is a Banach space endowed with the norm $\|u\|_{\WnuOm}= \big(\|u\|^p_{L^{p} (\Omega)}+ |u|^p_{\WnuOm}\big)^{1/p}$ with $|u|^p_{\WnuOm}:=\mathcal{E}^1_\Omega(u)$.
%\begin{align*}
%\|u\|_{\WnuOm}= \Big(\|u\|^p_{L^{p} (\Omega)}+ \iil_{\Omega\Omega} \big|u(x)-u(y) \big|^p \, \nu (x-y) \mathrm{d}x\,\mathrm{d}y\Big)^{1/p}. 
%\end{align*}
%%
For the standard example $\nu(h)= |h|^{-d-sp}$, $s\in (0,1)$, one recovers the Sobolev space of fractional order denoted $W^{s,p}(\Omega)$;  see \cite{Hitchhiker,Fog20} for more. If $\nu$ has full support, the space $ \WnuOmR= \big\lbrace u: \R^d \to \R ~\text{meas.} : u\in L^p(\Omega)\,\,\text{and}\,\, |u|^p_{\WnuOmR}<\infty \big\rbrace$, $|u|^p_{\WnuOmR}=\mathcal{E}^3_\Omega(u)$, is a Banach space with the norm $\|u\|_{\WnuOmR}= \big( \|u\|^p_{L^{p} (\Omega)}+ |u|^p_{\WnuOmR}\big)^{1/p}$. See \cite{FGKV19,Fog21b,Fog20} for recent results involving this types of spaces. 
%
%\begin{align*}
%\|u\|_{\WnuOmR} &:= \Big( \|u\|^p_{L^{p} (\Omega)}+ \iil_{\Omega\R^d} \big|u(x)-u(y) \big|^p \, \nu (x-y) \mathrm{d}x \, \mathrm{d}y\Big)^{1/p}. 
%\end{align*}
%\begin{align*}
%\iil_{\Omega\R^d} \big|u(x)-u(y) \big|^p \, \nu (x-y) \mathrm{d}x \, \mathrm{d}y \asymp\iil_{(\Omega^c\times \Omega^c)^c} \big|u(x)-u(y) \big|^p \, \nu (x-y) \mathrm{d}x \, \mathrm{d}y. 
%\end{align*}and we have 
%\begin{align*}
%\iil_{(\Omega^c\times \Omega^c)^c} \big|u(x)-u(y) \big|^p \, \nu (x-y) \mathrm{d}x \, \mathrm{d}y &= \iil_{\R^d\R^d} \big|u(x)-u(y) \big|^p \,\max(\mathds{1}_\Omega(x), \mathds{1}_\Omega(y) )\nu (x-y) \mathrm{d}x \, \mathrm{d}y\, , \\
%\iil_{\Omega\R^d} \big|u(x)-u(y) \big|^p \, \nu (x-y) \mathrm{d}x \, \mathrm{d}y &= \frac{1}{2} \iil_{\R^d\R^d} \big|u(x)-u(y) \big|^p \, \big[\mathds{1}_\Omega(x)+\mathds{1}_\Omega(y) \big]\nu (x-y) \mathrm{d}x \, \mathrm{d}y\,.
%\end{align*}
\noindent We recall that,  $\frac{1}{2}\mathcal{E}^2_\Omega(u)\leq \mathcal{E}^3_\Omega(u)\leq \mathcal{E}^2_\Omega(u)$.
%since $ \max\big(\mathds{1}_\Omega(x),\mathds{1}_\Omega(y) \big)\leq \big[\mathds{1}_\Omega(x)+\mathds{1}_\Omega(y) \big]\leq 2\max\big(\mathds{1}_\Omega(x),\mathds{1}_\Omega(y) \big)$. 
\noindent It is noteworthy to mention that, the space $\big(\WnuOmR, \|\cdot\|_{\WnuOmR}\big)$ is the core energy space for a large class of nonlocal problems with Dirichlet, Neumann or Robin boundary conditions. See for instance \cite{FKV15,FK22,DFK22, DROV17, Ros16}. If $\Omega\subset \R^d$ has a sufficiently regular boundary or $\Omega=\R^d$ then according to Theorem \ref{thm:BBM-limit-result} and Theorem \ref{thm:collapsing-accross-boundary}, it is legitimate to say that the nonlocal spaces $\big (W_{\nu_\eps}^{p}(\Omega|\R^d), \|\cdot\|_{W_{\nu_\eps}^{p}(\Omega|\R^d)}\big)_\eps$ and $\big(W_{\nu_\eps}^{p}(\Omega), \|\cdot\|_{W_{\nu_\eps}^{p}(\Omega)}\big)_\eps$ converge to the Sobolev space $\big(W^{1,p}(\Omega), \|\cdot\|^*_{W^{1,p}(\Omega)}\big)$ and  $\big(BV(\Omega), \|\cdot\|^*_{BV(\Omega)}\big)$,  where
\begin{align*}
\|u\|^{*}_{W^{1,p}(\Omega)}=\big( \|u\|^{p}_{L^{p}(\Omega)}+K_{d,p}\|\nabla u\|^{p}_{L^p(\Omega) }\big)^{1/p}\quad\text{and}\quad \|u\|^{*}_{BV(\Omega)} = \|u\|_{L^{1}(\Omega)} +K_{d,1}|u|_{BV(\Omega) }. 
\end{align*}

\smallskip

 Let us recall the following standard approximation result for the space $BV(\Omega)$; see \cite[p.172]{EG15},\cite[Theorem 14.9]{Leo17} or  \cite[Theorem 3.9]{AFP00}. 
 \begin{theorem}%[\hspace*{-1ex}\textnormal{\cite[p.172]{EG15},\cite[Theorem 14.9]{Leo17}, \cite[Theorem 3.9]{AFP00}}]
 \label{thm:meyer-serrin-bv}
 Let $\Omega\subset\R^d$ be open and $u \in BV(\Omega)$. There is  a sequence $(u_n)_n$ in $ BV(\Omega)\cap C^\infty(\Omega)$ such that $\|u_n-u\|_{L^1(\Omega)}\xrightarrow{n \to \infty}0$ and $\|\nabla u_n\|_{L^1(\Omega)}\xrightarrow{n \to \infty}|u|_{BV(\Omega)} $. 
 \end{theorem}

 %
 %\begin{remark} 
 \noindent Warning: the above approximation theorem does not claim that $|u_n-u|_{BV(\Omega)}\xrightarrow{n\to \infty}0$ but rather implies that $\|u_n\|_{W^{1,1}(\Omega)}\xrightarrow{n \to \infty}\|u\|_{BV(\Omega)}$. Strictly speaking, $BV(\Omega)\cap C^\infty(\Omega)$ is not necessarily dense in $BV(\Omega)$. 
 Recall that, if a function $u\in L^1(\Omega)$ is regular enough, say, $u \in W^{1,1}(\Omega)$ then we have $u\in BV(\Omega)$. From this we find that $BV(\Omega)\cap C^\infty(\Omega)= W^{1,1}(\Omega)\cap C^\infty(\Omega)$. 
 %\end{remark}
  
\smallskip 

 \noindent Next, we establish some useful estimates. Note that for $h\in \R^d$ we have 
\begin{align*}
\int_{\R^d} |u(x+h)-u(x)|^p\d x \leq 2^p\|u\|^p_{L^p(\R^d)}.
\end{align*}
\noindent Furthermore, using the density of $C_c^\infty(\R^d)$ in $W^{1,p}(\R^d)$ we find that 
\begin{align*}
\int_{\R^d} |u(x+h)-u(x)|^p\d x = \int_{\R^d} \Big|\int_0^1\nabla u(x+th)\cdot h\Big|^p\d x\leq |h|^p \|\nabla u\|^p_{L^{p}(\R^d)}. 
\end{align*}
Therefore, for every $u\in W^{1,p}(\R^d)$ and $h\in \R^d$ we have 
\begin{align}\label{eq:levy-p-estimate}
\int_{\R^d} |u(x+h)-u(x)|^p\d x\leq 2^p(1\land |h|^p)\|u\|^p_{W^{1,p}(\R^d)}.
\end{align}
\noindent By Theorem \ref{thm:meyer-serrin-bv} the $BV$-norm of an element in $BV(\R^d)$ can be approximated by the $W^{1,1}$-norms of elements in $W^{1,1}(\R^d)$. Whence for $p=1$, \eqref{eq:levy-p-estimate} implies that,  for $u\in BV(\R^d)$ and $h\in \R^d$  
\begin{align}\label{eq:levy-estimate-Bv}
\int_{\R^d} |u(x+h)-u(x)|\d x\leq 2(1\land |h|)\|u\|_{BV(\R^d)}. 
\end{align}

\smallskip

\begin{lemma}\label{lem:boundedness-limsup}
Assume $\nu:\R^d\setminus\{0\}\to [0,\infty) $ is $p$-L\'{e}vy integrable and $\Omega\subset \R^d$ is a $W^{1,p}$-extension domain (resp. $BV$-extension domain). 
There is  $ C= C(\Omega,d,p)>0$ independent of $\nu$ such that 
	\begin{align*}
	&\iint\limits_{\Omega\Omega}|u(x)-u(y)|^p\nu(x-y) \mathrm{d}y\mathrm{d}x \leq C \|u\|^p_{W^{1,p}(\Omega) } \|\nu\|_{L^{1}(\R^d, 1\land |h|^p) },\quad \text{for all $u\in W^{1, p}(\Omega)$}\\
		(\text{resp.}\quad
	&\iint\limits_{\Omega\Omega}|u(x)-u(y)|\nu(x-y) \mathrm{d}y\mathrm{d}x \leq C \|u\|_{BV(\Omega) } \|\nu\|_{L^{1}(\R^d, 1\land |h|) }, \quad\, \text{for all}\quad u \in BV(\Omega)).
	\end{align*}
\end{lemma}

\smallskip

\begin{proof}
	Let $\overline{u}$ be a $ W^{1,p}$-extension of $u$ on $\R^d.$ The estimate \eqref{eq:levy-p-estimate} implies 
	\begin{alignat*}{2}
	&\iint\limits_{\Omega\Omega}|u(x)-u(y)|^p\nu(x-y) \mathrm{d}y\mathrm{d}x 
	&&\leq \iint\limits_{\R^d\R^d}|\overline{u}(x+h)-\overline{u}(x)|^p \nu(h) \mathrm{d}h\mathrm{d}x\\
	&= \int_{\R^d} \nu(h) \mathrm{d}h \int_{ \R^d} |\overline{u}(x+h)-\overline{u}(x)|^p\mathrm{d}x
	&&%\leq \|\overline{u}\|^p_{W^{1,p}(\R^d)} \int\limits_{\R^d}2^p(1\land |h|^p) \nu(h) \mathrm{d}h \\
	\leq C\|u\|^p_{W^{1,p}(\Omega) } \|\nu\|_{L^{1}( \R^d, 1\land |h|^p) }. 
	\end{alignat*}
	\noindent Likewise, if $p=1$ and $u\in BV(\Omega)$ one gets the other estimate from the estimate \eqref{eq:levy-estimate-Bv}. 
\end{proof}
\smallskip

\noindent An immediate consequence of Lemma \ref{lem:boundedness-limsup} is the following embedding result. 

\begin{theorem}\label{thm:w1p-in nonlocal} 
%Assume $\nu:\R^d\setminus\{0\}\to [0,\infty) $ is $p$-L\'{e}vy integrable.
Assume $\nu\in L^1(\R^d, 1\land|h|^p)$ with $p\geq1$ and $\Omega\subset \R^d$ is a $W^{1,p}$-extension domain. There holds that the embedding $W^{1,p}(\Omega)\hookrightarrow \WnuOm$ is continuous. 
Furthermore, for $p=1$ and if $\Omega$ is a $BV$-extension domain then the embedding $BV(\Omega)\hookrightarrow W^1_\nu(\Omega)$  is also continuous. 
%If is a $W^{1,p}$-extension domain then the embedding $W^{1,p}(\Omega)\hookrightarrow \WnuOm$ is continuous and \, (resp. $BV(\Omega)\hookrightarrow W^1_\nu(\Omega), $p=1$)$ is continuous. 
\end{theorem}

\smallskip 

\noindent It is worth emphasizing that the above embeddings may fail if $\Omega$ is not an extension domain (see the counterexample \ref{Ex:counterexample-extension}). Another straightforward consequence of Lemma \ref{lem:boundedness-limsup} is the following. 

\begin{theorem}\label{thm:limpsup}
Let $\Omega$ be a $W^{1,p}$-extension domain, $p\geq1$. There is  $C=(\Omega, d,p)>0$ such that 
	\begin{align*}
	\limsup_{\varepsilon\to 0} \iil_{\Omega \Omega } |u(x)-u(y)|^p\nu_\varepsilon(x-y)\d x\d y \leq C \|u\|^p_{W^{1,p}(\Omega)}\qquad\text{for all $u\in W^{1,p}(\Omega)$}.
	\end{align*}
	If $p=1$ and $\Omega$ is a $BV$-extension domain we also have ,
	\begin{align*}
	\limsup_{\varepsilon\to 0} \iil_{\Omega \Omega } |u(x)-u(y)|\nu_\varepsilon(x-y)\d x\d y \leq C\|u\|_{BV(\Omega)}\qquad\text{for all $u\in BV(\Omega)$}.
	\end{align*}
\end{theorem}

\smallskip

\noindent The next proposition shows that the $p$-L\'{e}vy integrability condition is consistent and optimal in the sense that it draws a borderline for which a space of type $\WnuOm$ is trivial or not. 
\begin{proposition}\label{prop:borderline} Let $\nu: \R^d\to [0,\infty]$ be symmetric. The following assertions are true. 
\begin{enumerate}[$(i)$]
\item If $\nu\in L^1(\R^d)$ then $\WnuOm= L^p(\Omega)$ and $\WnuOmR\cap L^p(\R^d)= L^p(\R^d)$.
		
\item If $\nu \in L^1(\R^d, 1\land |h|^p)$ then $W^{1,p}(\R^d)\subset W_\nu^p(\R^d)$, hence the spaces $\WnuOm$ and $\WnuOmR$ contain $C_c^\infty(\R^d)$. Moreover, if $\Omega$ is bounded, 
then the spaces $\WnuOm$ and $\WnuOmR$  also contain bounded Lipschitz functions. 

\item Assume $\nu$ is radial, $\Omega$ is connected and put $C_\delta=\int_{B_\delta(0)} \hspace*{-0.5ex}|h|^p\nu(h)\d h$. If $C_\delta=\infty$  for all $\delta>0$, %in particular $\nu \not\in L^1(\R^d, 1\land |h|^p)$. 
then any  function $u\in W^{1,p}(\Omega)\cap \WnuOm$ or $u\in C^1(\Omega)\cap \WnuOm$ is a  constant function.
\item Assume $\nu \in L^1(\R^d, 1\land |h|^p)$ and $\nu$ is radial. Given $u\in W^{1,p} (\R^d)$ there is $\delta=\delta(u)>0$ so that 
%depending on $u$ such that
\begin{align}\label{eq:pseudo-equiv}
2^{-p}K_{d,p} C_\delta\|\nabla u\|^p_{L^{p} (\R^d)} \leq |u|^p_{W^p_\nu (\R^d)}\leq  2^p\|\nu\|_{L^1(\R^d, 1\land |h|^p\d h)}) \|u\|^p_{W^{1,p} (\R^d)}.
\end{align}
	\end{enumerate}
\end{proposition}
%\noindent \textbf{Warning!} The equivalence \eqref{eq:equiv-sobolev} does not imply that $W^{1,p}(\R^d) = W^p_\nu(\R^d)$. 

\smallskip 

\begin{proof} 
\noindent $(i)$ is obvious.
$(ii)$ For a bounded Lipschitz function $u$, we have $|u(x) -u(y)|^p\leq C(1\land|x-y|^p)$ for some $C>0$. Hence if $\Omega$ is bounded, by integrating both sides it follows that $u\in \WnuOm$ and $u\in \WnuOmR$. The inclusion $W^{1,p}(\R^d)\subset W_\nu^p(\R^d)$ follows from Lemma \ref{lem:boundedness-limsup} or from the estimate \eqref{eq:levy-p-estimate}. $(iii)$ Let $u\in W^{1,p}(\Omega)\cap \WnuOm$ or $u\in C^1(\Omega)\cap \WnuOm$ and let $K \subset\Omega$ be a compact set. Since $|\nabla u|\in L^p(K)$, for arbitrary $\eta>0$ there is $0<\delta= \delta(\eta,K)<\dist(K, \partial \Omega)$ such that, 
\begin{align*}
\|\nabla u(\cdot+h)-\nabla u\|_{L^p(K)}<\eta\quad \text{for all $|h|\leq \delta $.} 
\end{align*} 
%So that using the inequality $|a+b|^p\leq 2^{p-1}(|a|^p+|b|^p)$ for all $a,b\in\R$, we find that
Minkowski's inequality implies 

\begin{align*}
\Big(\int_K\int_{B_\delta(0)}	|\nabla u(x)\cdot h|^p\nu(h)\d h\d x\Big)^{1/p} &\leq \Big(\int_K\int_{B_\delta(0)}	\hspace{-0.5ex} \Big|\int_0^1 \nabla u(x+th) \cdot h\d t\Big|^p\nu(h) \d h\d x\Big)^{1/p}+\eta C_\delta^{1/p}.
%\\&+ \varepsilon\Big( \int_{B_\delta(0)} \hspace{-2ex}|h|^p\nu(h)\d h\Big)^{1/p}.
\end{align*} 
%	\begin{align*}
%	2^{1-p}	\int_K\int_{B_\delta(0)}	|\nabla u(x)\cdot h|^p\nu(h)\d h\d x &\leq \int_K\int_{B_\delta(0)}	\hspace{-1ex} \Big|\int_0^1 \nabla u(x+th) \cdot h\d t\Big|^p\nu(h) \d h\d x+\eta \int_{B_\delta(0)} \hspace{-2ex}|h|^p\nu(h)\d h.
%	\end{align*}
The choice $0<\delta<\dist(K, \partial \Omega)$ ensures that $B_\delta(x)\subset \Omega$ for all $x\in K$. From the foregoing, using the fundamental theorem of calculus, polar coordinates and the formula \eqref{eq:rotation-invariant-constant} yield 
	\begin{align*}%\label{eq:optimal-non-levy}
	\begin{split}
		|u|_{\WnuOm}
		%&\geq 
	%\iil_{\Omega\Omega}|u(x)-u(y)|^p\nu(x-y)\,\d y\,\d x
	&\geq\Big( \int_{K} \int_{B_\delta(0)} \Big| \int_0^1\nabla u(x+th)\cdot h\d t\Big|^p \nu(h) \d h\,\d x\Big)^{1/p}\\
	%&\geq 2^{1-p}\int_{K} \int_{B_\delta(0)} |\nabla u(x)\cdot h|^p \nu(h) \d h\,\d x- \eta \int_{B_\delta(0)} |h|^p\nu(h)\d h\\
	&\geq \Big( \int_{K} \int_{\mathbb{S}^{d-1}} \hspace{-2ex} |\nabla u(x)\cdot w|^p \mathrm{d}\sigma_{d-1}(w) \int_{0}^{\delta}\hspace{-1ex} r^{p+d-1} \nu(r)\d r
	\Big)^{1/p}- \eta \Big( \int_{B_\delta(0)} \hspace{-2ex}|h|^p\nu(h)\d h\Big)^{1/p}\\
	&= \Big(K^{1/p}_{d,p}\|\nabla u\|_{L^p(K)}-\eta\Big)\Big( \int_{B_\delta(0)} \hspace{-2ex}|h|^p\nu(h)\d h\Big)^{1/p}. 
	\end{split}
	\end{align*}
	Therefore, for each $\eta>0$ and each compact set $K\subset \Omega$ we have 
	\begin{align}\label{eq:optimal-non-levy}
	\begin{split}
	|u|_{\WnuOm}&\geq C^{1/p}_\delta\big(K^{1/p}_{d,p}\|\nabla u\|_{L^p(K)}-\eta\big).
	\end{split}
	\end{align}
	
	\noindent Since $\|u\|_{\WnuOm}<\infty$ and $C_\delta= \infty$, this is possible only if $\|\nabla u\|^p_{L^p(K)}=0$. As the compact set $K\subset \Omega$ is arbitrary, we find that $\nabla u=0$ a.e. on $\Omega$. 
	Thus $u$ is a constant since $\Omega$ is connected. $(iv)$ The upper inequality clearly follows from \eqref{eq:levy-p-estimate}. Proceeding as for the estimate \eqref{eq:optimal-non-levy} by taking  $\Omega=\R^d$  and $ K=\R^d$ also yields that, for all $\eta>0$ there is $\delta= \delta(\eta)>0$ such that 
 	\begin{align}\label{eq:optimal-non-levy-bis}
	\begin{split}
|u|_{W^p_\nu(\R^d)}\geq C^{1/p}_\delta\big(K^{1/p}_{d,p}\|\nabla u\|_{L^p(\R^d)}-\eta\big).
	\end{split}
	\end{align}
If $\|\nabla u\|_{L^p(\R^d)}\neq 0$, taking $\eta=\frac{1}{2}K^{1/p}_{d,p}\|\nabla u\|_{L^p(\R^d)}$ yields $|u|^p_{W^p_\nu(\R^d)} \geq 2^{-p}K_{d,p}C_\delta\|\nabla u\|^p_{L^p(\R^d)}.$
%	\begin{align*}
%	|u|^p_{W^p_\nu(\R^d)} &\geq 2^{-p}K_{d,p}C_\delta\|\nabla u\|^p_{L^p(\R^d)}. 
%	\end{align*}
This estimate remains true  for any $\delta>0$, if  $\|\nabla u\|_{L^p(\R^d)}=0$. 
\end{proof}

\noindent The next theorem provides a characterization of the $p$-L\'{e}vy integrability condition. 
\begin{theorem}
	Assume $\nu:\R^d\to [0,\infty]$ is radial. The following assertions are equivalent. 
	\begin{enumerate}[$(i)$]
		\item The $p$-L\'{e}vy integrability condition in \eqref{eq:plevy-cond} holds.
		\item The embedding $W^{1,p}(\R^d)\hookrightarrow W_\nu^p(\R^d)$ is continuous. 
		\item $\mathcal{E}^1_{\R^d}(u)<\infty$ for all $u\in W^{1,p}(\R^d)$. 
		\item $\mathcal{E}^1_{\R^d}(u)<\infty$ for all $u\in C_c^\infty(\R^d)$. 
		\item   There exists $u\in C_c^\infty(B_1(0))\setminus\{0\}$ such that $\mathcal{E}^1_{\R^d}(u_n)<\infty$ for all $n\geq 1$, $u_n(x)=n^d u(nx)$. 
		% for some $u\in C_c^\infty(B_1(0))$ with $u\not\equiv 0$. 
	\end{enumerate}
This remains true when $p=1$ with $BV(\R^d)$ in place of $W^{1,1}(\R^d)$.
\end{theorem}

\begin{proof}
	$(i)\implies (ii)$. The right hand side of the estimate \eqref{eq:pseudo-equiv}  implies the continuity of the embedding  $ W^{1,p}(\R^d)\hookrightarrow W^p_\nu(\R^d)$. The implications $(ii)\implies (iii)$, $(iii)\implies (iv)$ and $(iv)\implies (v)$ are straightforward. 
	Let us prove that $(v)\implies (i)$. 
Given that $u\in C_c^\infty(B_1(0)\setminus\{0\}$ we have  $\|\nabla u\|_{L^{p} (\R^d)}\neq 0$. By Proposition \ref{prop:borderline} $(iv)$ there exists $\delta=\delta(u)>0$ (see the estimate \eqref{eq:pseudo-equiv}) such that $ \mathcal{E}^1_{\R^d} (u)\geq 2^{-p} K_{d,p} C_\delta\|\nabla u\|^p_{L^{p} (\R^d)}$ and hence 
	$C_\delta=\int_{B_\delta(0)}|h|^p\nu(h)\d h<\infty$. Next, we fix $n\geq 1$ such that $\delta>\frac{2}{n}$ so that  $\supp u_n\subset B_{\delta/2}(0)$. Since  $B_{\delta/2}(x)\subset B_{\delta}(0)$ for all $x\in B_{\delta/2}(0)$ we have 
\begin{align*}
\infty>\mathcal{E}^1_{\R^d} (u_n) \geq 2\int_{B_{\delta/2}(0)}|u_n(x)|^p\int_{\R^d\setminus B_{\delta/2}(0)}\nu(x-y)\d y\d x\geq 2\|u_n\|^p_{L^p(\R^d)} \int_{\R^d\setminus B_{\delta}(0)}\nu(h)\d h.
\end{align*}
Thus $\int_{|h|\geq \delta}\nu(h)\d h<\infty$. Accordingly $\nu\in L^1(\R^d, 1\land |h|^p).$
The case $p=1$ follows analogously. 
\end{proof}

%%%%%%%%%%%%%%%%%%%%%%%%%%%%%%%%%%%%%%%%%%%%%%%%%%%%%%%%%%%%%%%%%%%%%%
\section{Main results}\label{sec:charact-W1p}
%%%%%%%%%%%%%%%%%%%%%%%%%%%%%%%%%%%%%%%%%%%%%%%%%%%%%%%%%%%%%%%%%%%%%

%% the reflexivity property of $L^p$-space for $1<p<\infty$, namely the Riesz representation Theorem The %%result infers that if $1\leq p<\infty$ then $L^{p'}(\Omega)$ is isomorphic to the dual of $L^p(\Omega)$ and any linear continuous functional on $L^p(\Omega)$ is of the form $u\mapsto\int_\Omega g u$ with $g\in L^{p'}\Omega)$. 

\noindent First and foremost, the proof of Theorem \ref{thm:BBM-limit-result} in the case $\Omega=\R^d$ is much simpler. Indeed, by the estimates \eqref{eq:optimal-non-levy-bis}  and \eqref{eq:split-estimate} below, for sufficiently small $\eta>0$,  there is $\delta= \delta(\eta)>0$ such that 
	\begin{align*}
&\iil_{\R^d \R^d} \hspace{-1ex}|u(x)-u(y)|^p\nu_\eps(x-y)\d y\d x
\geq \big(K^{1/p}_{d,p}\|\nabla u\|_{L^p(\R^d)}-\eta\big)^p
\hspace*{-1ex}\int_{B_\delta(0)} \hspace{-2ex}|h|^p\nu_\eps(h)\d h, \\
%\xrightarrow{\eps, \eta \to 0} K_{d,p} \|\nabla u\|^p_{L^p(\R^d)}, 
&\iil_{\R^d\R^d} \hspace{-1ex}|u(x)-u(y)|^p\nu_\varepsilon(x-y) \mathrm{d}y \d x\leq K_{d,p}\|\nabla u\|^p_{L^p(\R^d)}+ 2^p\|u\|^p_{L^p(\R^d)}\int_{|h|>\delta}\hspace{-2ex}\nu_\varepsilon(h)\,\d h.
%\xrightarrow{\eps\to0} K_{d,p} \|\nabla u\|^p_{L^p(\R^d)}. 
\end{align*}
\noindent Letting $\eps\to 0$ and $\eta\to0$ successively, using the formulas \eqref{eq:concentration-bis} and \eqref{eq:concentration-prop}, we get 
 \begin{align}\label{eq:lim-full-form}
 \lim_{\eps\to 0}\iil_{\R^d\R^d} \hspace{-1ex}|u(x)-u(y)|^p\nu_\varepsilon(x-y) \d y \d x
 = K_{d,p}\|\nabla u\|^p_{L^p(\R^d)}. 
 %\liminf_{\eps\to 0}\iil_{\R^d\R^d} \hspace{-1ex}|u(x)-u(y)|^p\nu_\varepsilon(x-y) \d y \d x. 
 \end{align}
\noindent The case $p=1$ and $u\in BV(\R^d)$ can be proved analogously. In fact, it can be shown that  \eqref{eq:lim-full-form} holds if and only if up to a multiple factor $(\nu_\eps)_\eps$ satisfies \eqref{eq:plevy-approx}. In other words, the class $(\nu_\eps)_\eps$ is the largest (the sharpest) class for which the BBM formula \eqref{eq:lim-full-form} holds. From now on, we assume $\Omega\neq \R^d$. We start with the following lemma which is somewhat a revisited version of \cite[Lemma 1]{BBM01}.

\begin{lemma}\label{lem:boun-integration-bypart} Assume $\nu \in L^1(\R^d, 1\land |h|^p )$ is symmetric, $p\geq1$.  Given $ u\in L^p(\R^d)$, $\varphi \in C_c^\infty(\R^d)$ and a unit vector $e\in \mathbb{S}^{d-1}$ we have 
	\begin{relsize}{-0}
		\begin{align*}
		\Big | \hspace{-3ex}\iint\limits_{(y-x)\cdot e\geq 0} \hspace{-3ex} u(x) \frac{\varphi(y) - \varphi(x)}{|x-y|}(1\land |x-y|^p) \nu(x-y) \mathrm{d}y &\mathrm{d}x\Big |+ \hspace{-0.2ex}\Big | \hspace{-2ex}\iint\limits_{(y-x)\cdot e\leq 0} \hspace{-3ex} u(x) \frac{\varphi(y) - \varphi(x)}{|x-y|}(1\land |x-y|^p) \nu(x-y) \mathrm{d}y \mathrm{d}x\Big |\\
		&\leq \hspace{-1ex}\iint\limits_{\R^d\R^d} \frac{|u(x)-u(y)|}{|x-y|}|\varphi(x)|(1\land |x-y|^p) \nu(x-y) \mathrm{d}y \mathrm{d}x. 
		\end{align*}
	\end{relsize}
	
\end{lemma}

\begin{proof}
	Let us introduce the truncated measure $\widetilde{\nu}_\delta (h)= |h|^{-1}(1\land |h|^p) \nu(h)~\mathds{1}_{\R^d\setminus B_\delta }(h)$ for $\delta>0$ which enables us to rule out an eventual singularity of $\nu$ at the origin. Moreover, note that $\widetilde{\nu}_\delta \in L^1(\R^d)$. It turns out that the mappings $(x,y)\mapsto u(x) \varphi(y) \widetilde{\nu}_\delta (x-y)$ and $(x,y)\mapsto u(x) \varphi(x) \widetilde{\nu}_\delta (x-y)$ are integrable. Indeed, using H\"older inequality combined with Fubini's theorem yield
	%
	%	\begin{relsize}
	\begin{align*}
	&\iil_{\R^d\R^d} |u(x) \varphi(x)| \widetilde{\nu}_\delta (x-y)\d y \,\d x 
	= \iil_{{ |x-y|\geq \delta}} |u(x) \varphi(x)| |x-y|^{-1}(1\land |x-y|^p) \nu(x-y) \,\d x \,\d y \\
	&\leq \delta^{-1} \Big(\hspace{-2ex}\iil_{{ |x-y|\geq \delta}} \hspace{-2ex}| u(x)|^p (1\land |x-y|^{p})\nu(x-y) \,\d y \,\d x \Big)^{1/p} 
	\Big(\hspace{-2ex}\iil_{{ |x-y|\geq \delta}} \hspace{-2ex} |\varphi(x)|^{p'}(1\land |x-y|^p) \nu(x-y) \,\d y \,\d x\Big)^{1/p'} \\
	&\leq \delta^{-1}\|\varphi\|_{L^{p'}(\R^d)}\|u\|_{L^p(\R^d)}\int_{\R^d}(1\land |h|^p)\nu(h)\,\d h<\infty.
	\end{align*}
	%\end{relsize}
	\noindent Analogously, we also get
	\begin{align*}
	\iil_{\R^d \R^d} |u(x) \varphi(y)| \widetilde{\nu}_\delta (x-y)\,\d y \,\d x 
	&\leq \delta^{-1}\|\varphi\|_{L^{p'}(\R^d)}\|u\|_{L^p(\R^d)}\int_{\R^d}(1\land |h|^p)\nu(h)\,\d h<\infty.
	\end{align*}
	Consequently by interchanging $x$ and $y$, using Fubini's theorem and the symmetry of $\nu$ we obtain 
	\begin{alignat*}{2}
	& \iil_{(y-x)\cdot e\geq 0} u(x) \varphi(x) \widetilde{\nu}_\delta (x-y)\,\d y \,\d x 
	& &= \iil_{(x-y)\cdot e\geq 0} u(y) \varphi(y)\widetilde{\nu}_\delta (x-y)\,\d x \,\d y  \\
%	&= \int_{\R^d} u(y)\varphi(y)\,\d y \il_{h\cdot e\geq 0} \widetilde{\nu}_\delta (h) \,\d h
%	& &= \int_{\R^d} u(y)\varphi(y)\,\d y \il_{h\cdot e\leq 0} \widetilde{\nu}_\delta (h) \,\d h\\
%	&= \int_{\R^d} u(y)\varphi(y)\,\d y \il_{h\cdot e\geq 0} \widetilde{\nu}_\delta (h) \,\d h
	 &= \iil_{(y-x')\cdot e\geq 0} u(y)\varphi(y)\widetilde{\nu}_\delta (y-x') \,\d y\,\d x' && \qquad(x'= 2y-x, \d x= \d x'). 
	\end{alignat*}
	Therefore we have 
	\begin{relsize}{-0}
		\begin{align*}
		\hspace{-2ex} \Big|\int_{\R^d} u(x)\,\d x \hspace{-3ex}\il_{(y-x)\cdot e\geq 0} \hspace*{-2ex}(\varphi(y)&-\varphi(x)) \widetilde{\nu}_\delta (x-y)\,\d y\Big| \\
		&=\Big |\hspace{-2ex}\iil_{(y-x)\cdot e\geq 0} \hspace*{-2ex} u(x) \varphi(y) \widetilde{\nu}_\delta (x-y)\,\d y \,\d x -\hspace{-2ex} \iil_{(y-x)\cdot e\geq 0} \hspace*{-2ex} u(x) \varphi(x) \widetilde{\nu}_\delta (x-y)\,\d y \,\d x \Big| \\
		&=\Big | \int_{\R^d} \varphi(y) \,\d y \hspace{-2ex}\il_{(y-x)\cdot e\geq 0} \hspace*{-2ex} (u(x)-u(y)) \widetilde{\nu}_\delta (x-y) \,\d x\Big|
		\\& \leq \int_{\R^d} |\varphi(y)|\,\d y \il_{(y-x)\cdot e\geq 0} \hspace*{-3ex} \frac{|u(x)-u(y)|}{|x-y|} (1\land |x-y|^p)\nu(x-y) \,\d x\\
		&= \int_{\R^d} |\varphi(x)|\,\d x \il_{(y-x)\cdot e\leq 0} \hspace*{-3ex} \frac{|u(y)-u(x)|}{|x-y|} (1\land |x-y|^p)\nu(x-y) \,\d y.
		\end{align*}
	\end{relsize}
	Thus letting $\delta \to 0$ implies 
	%\begin{relsize}{-0.5}
		\begin{align}\label{eq:delta-estimate1}
		\begin{split}
	\Big| \hspace{-2ex}\iil_{(y-x)\cdot e\geq 0} \hspace{-1ex} u(x)(\varphi(y)-\varphi(x)) &|x-y|^{-1}(1\land |x-y|^p)\nu(x-y) \,\d y \,\d x\Big| 
	\\&\leq \hspace{-3ex}\iil_{(y-x)\cdot e\leq 0} \hspace*{-3ex}|\varphi(x)| \frac{|u(y)-u(x)|}{|x-y|}(1\land |x-y|^p)\nu(x-y) \,\d y\,\d x.
		\end{split}
		\end{align}
%	\end{relsize}
	%
	Likewise one has
	%
	%\begin{relsize}{-0.5}
		\begin{align}\label{eq:delta-estimate2}
		\begin{split}
	\Big| \hspace{-2ex}\iil_{(y-x)\cdot e\leq 0} \hspace{-2ex} u(x) (\varphi(y)-\varphi(x)) &|x-y|^{-1}( 1\land |x-y|^p)\nu(x-y)\,\d y \,\d x\Big| \\
	&\leq \hspace{-2ex} \iil_{(y-x)\cdot e\geq 0} \hspace*{-3ex}|\varphi(x)| \frac{|u(y)-u(x)|}{|x-y|} (1\land |x-y|^p)\nu(x-y) \,\d y\,\d x.
		\end{split}
		\end{align}
%	\end{relsize}
	%
Summing the estimates \eqref{eq:delta-estimate1} and \eqref{eq:delta-estimate2} gives the desired inequality.
\end{proof}

\smallskip

\begin{theorem}\label{thm:bounded-liminf-estimate} Let $\Omega\subset \R^d$ be an open, $u \in L^p(\Omega)$, $p\geq 1$ and $A_p$ be defined as in\eqref{eq:liminf-condition}.
	 %Then for every  with $1\leq p<\infty $, 
Then given a unit vector $e\in \mathbb{S}^{d-1}$ and $\varphi\in C_c^\infty(\R^d)$ with support in $\Omega$ the following estimate holds true 
	\begin{align}\label{eq:cont-weak-derivatve}
	\left|\int_{\Omega} u(x) \nabla \varphi (x)\cdot e ~\mathrm{d}x \right| \leq \frac{A_p^{1/p}}{ K_{d,1}}\|\varphi\|_{L^{p'}(\Omega)}.
	\end{align} 
\end{theorem}

\smallskip 

\begin{proof}
Throughout, to alleviate the notation we denote $\pi_\varepsilon(x-y)= (1\land |x-y|^p)\nu_\varepsilon(x-y)$. Let $\overline{u}\in L^p(\R^d)$ be the zero extension of $u$ off $\Omega$. Since $\supp\varphi\subset \Omega,$ we have the identity
\begin{align*}
\int_{\R^d} |\varphi(x)|\,\d x \int_{\R^d} \frac{|\overline{u}(y)-\overline{u}(x)|}{|x-y|} \pi_\varepsilon(x-y) \,\d y
&= \iil_{\Omega \Omega}\frac{|u(y)-u(x)|}{|x-y|} |\varphi(x)| \pi_\varepsilon(x-y) \,\d y \,\d x \\
&+ \int_{\supp(\varphi)} |\varphi(x)|\,\d x \int_{\R^d\setminus \Omega } \frac{|u(x)|}{|x-y|} \pi_\varepsilon(x-y) \,\d y. 
\end{align*}
\noindent First, for $\delta= \operatorname{dist}(\supp(\varphi), \partial\Omega)>0$, the H\"older inequality implies 
\begin{align*}
\hspace{-2ex}\int_{\supp(\varphi)} \hspace{-1ex}|\varphi(x)|\,\d x\hspace{-1ex} 
\int_{\R^d\setminus \Omega }\hspace{-0.5ex} \frac{|u(x)|}{|x-y|} \pi_\varepsilon(x-y) \,\d y
\leq \delta^{-1} \|u\|_{L^{p}(\Omega)} \|\varphi\|_{L^{p'}(\Omega)} \hspace{-1ex}
\int_{|h|\geq \delta } \hspace{-1ex}(1\land |h|^p)\nu_\varepsilon(h) \,\d h
\xrightarrow{\varepsilon\to 0}0. 
\end{align*}

\noindent Second, using again the H\"older inequality and $|h|^{-p}(1\land |h|^p)\leq 1$ we find that

\begin{align*}
\hspace{-1ex}\iil_{\Omega \Omega}\frac{|u(y)-u(x)|}{|x-y|} |\varphi(x)| 
&\pi_\varepsilon(x-y) \,\d y \,\d x \\
& \leq \Big( \iil_{\Omega \Omega}\frac{|u(y)-u(x)|^p}{|x-y|^p} \pi_\varepsilon(x-y) \,\d y \,\d x  \Big)^{1/p} \Big( \iil_{\Omega \Omega}|\varphi(x)|^{p'} \pi_\varepsilon(x-y) \,\d y \,\d x  \Big)^{1/p'}\\
&\leq \|\varphi\|_{L^{p'}(\Omega)}\Big( \iil_{\Omega \Omega}|u(y)-u(x)|^p\nu_\varepsilon(x-y) \,\d y \,\d x  \Big)^{1/p} . 
\end{align*}
\noindent Therefore inserting these two estimates  in the previous identity and combining the resulting estimate with that of Lemma \ref{lem:boun-integration-bypart} and that of the assumptions imply
%
%\begin{relsize}{-1}
\begin{align}\label{eq:nonlocal-liminf}
\begin{split}
&\liminf_{\varepsilon \to 0}\Big|\int_{\Omega} u(x)\,\d x \hspace{-3ex} \il_{(y-x)\cdot e\geq 0} \hspace{-3ex} \frac{(\varphi(y)-\varphi(x)) }{|x-y|} (1\land |x-y|^p)\nu_\varepsilon(x-y)\,\d y \Big|~+\\
&\liminf_{\varepsilon \to 0} \Big|\int_{\Omega} u(x)\,\d x \hspace{-3ex} \il_{(y-x)\cdot e\leq 0} \hspace{-3ex}\frac{(\varphi(y)-\varphi(x)) }{|x-y|}(1\land |x-y|^p)\nu_\varepsilon(x-y) \,\d y \Big|\leq A_p^{1/p} \|\varphi\|_{L^{p'}(\Omega)}.
\end{split}
\end{align}
	%\end{relsize}

\noindent It remains to compute the limits appearing on the left hand side of \eqref{eq:nonlocal-liminf}. For all $x, h\in \R^d$ we have 
\begin{align*}
	\varphi(x+h)-\varphi(x)= \nabla \varphi(x)\cdot h+ \int_{0}^1\big(\nabla\varphi(x+th)-\nabla\varphi(x)\big)\cdot h\d t
\end{align*}

\noindent and $|\nabla \varphi (x+h)- \nabla \varphi(x)|\leq C (1\land |h|)$. So that Remark \ref{rem:asymp-nu} implies

\begin{align*}
\lim_{\varepsilon \to 0} \il_{h\cdot e\geq 0} \int_0^1\Big|\big[\nabla \varphi(x+th)-\nabla\varphi(x)\big]\cdot \frac{h}{|h|}\Big|\d t (1\land |h|^p)\nu_\varepsilon(h) \d h\leq C\lim_{\varepsilon \to 0} \int_{\R^d} (1\land |h|^{p+1})\nu_\varepsilon(h) \,\d h=0. 
\end{align*}
Thus, using the above expression and the fact that $\int_{\R^d} (1\land |h|^p)\nu_\varepsilon(h) \,\d h=1$, we obtain
\begin{align*}
\lim_{\varepsilon \to 0} \il_{(y-x)\cdot e\geq 0} \hspace{-3ex} \frac{(\varphi(y)-\varphi(x)) }{|x-y|} &(1\land |x-y|^p)\nu_\varepsilon(x-y) \,\d y\\
&= \lim_{\varepsilon \to 0} \il_{\mathbb{S}^{d-1} \cap \{w\cdot e\geq 0\}} \hspace{-4ex}\nabla\varphi(x)\cdot w \,\d \sigma_{d-1}(w)\int_0^\infty (1\land r^p)r^{d-1}\nu_\varepsilon(r) \,\d r \\
&= |\mathbb{S}^{d-1}|^{-1}\il_{\mathbb{S}^{d-1} \cap \{w\cdot e\geq 0\}} \hspace{-3ex}\nabla\varphi(x)\cdot w \,\d \sigma_{d-1}(w). 
\end{align*}
\noindent Let $(e, v_2, \cdots v_d)$ be an orthonormal basis of $\R^d$ in which we write the coordinates $w= (w_1,w_2, \cdots, w_d)= (w_1,w')$ that is $w_1 = w\cdot e$ and $ w_i= w\cdot v_i$. Similarly, in this basis one has $\nabla \varphi (x) =(\nabla \varphi (x)\cdot e, (\nabla \varphi (x))' )$. Observe that $\nabla \varphi (x) \cdot w = \big(\nabla \varphi (x)\cdot e\big) (w\cdot e)+ [\nabla \phi (x)]'\cdot w' $. We find that

\begin{align*}
\il_{\mathbb{S}^{d-1} \cap \{w\cdot e\geq 0\}} \hspace{-3ex} \nabla \varphi (x) \cdot w \d \sigma_{d-1}(w) = \il_{\mathbb{S}^{d-1} \cap \{w\cdot e\geq 0\}}\hspace{-3ex} ( \nabla \varphi (x)\cdot e )( w\cdot e) \d \sigma_{d-1}(w)+ \il_{\mathbb{S}^{d-1} \cap \{w\cdot e\geq 0\}} \hspace{-3ex} (\nabla \varphi (x))' \cdot w' \d \sigma_{d-1}(w). 
\end{align*}
\noindent Consider the rotation $O(w) = (w_1, -w') = (w\cdot e, -w') $ then the rotation invariance of the Lebesgue measure entails that $\d \sigma_{d-1}(w) = \d \sigma(O(w)) $ and we have 
\begin{align*}
\il_{\mathbb{S}^{d-1} \cap \{w\cdot e\geq 0\}} \hspace{-3ex} (\nabla \varphi (x))' \cdot w' \d \sigma_{d-1}(w) = -\il_{\mathbb{S}^{d-1} \cap \{w\cdot e\geq 0\}} \hspace{-3ex} (\nabla \varphi (x))' \cdot w' \d \sigma_{d-1}(w) =0. 
\end{align*}
Whereas, by symmetry we have 
\begin{align*}
\il_{\mathbb{S}^{d-1} \cap \{w\cdot e\geq 0\}} \hspace{-3ex} w\cdot e \, \,\d \sigma_{d-1}(w) = -\il_{\mathbb{S}^{d-1} \cap \{w\cdot e\leq 0\}} \hspace{-3ex} w\cdot e \d \sigma_{d-1}(w) = \frac{ 1}{2} \il_{\mathbb{S}^{d-1}} |w\cdot e |\d \sigma_{d-1}(w) = \frac{ 1}{2}K_{d,1}. 
\end{align*}
\noindent Altogether, we find that 
\begin{align*}
|\mathbb{S}^{d-1}|^{-1}\il_{\mathbb{S}^{d-1} \cap \{w\cdot e\geq 0\}}\hspace{-4ex} \nabla\varphi(x)\cdot w \,\d \sigma_{d-1}(w)= \frac{\nabla\varphi(x)\cdot e}{2}\fint_{\mathbb{S}^{d-1}} | w \cdot e|\,\d \sigma_{d-1}(w) = \frac{ 1}{2}K_{d,1}\nabla \varphi (x)\cdot e. 
\end{align*}
In conclusion,
\begin{align}\label{eq:K1d}
&\lim_{\varepsilon \to 0} \il_{(y-x)\cdot e\geq 0} \hspace{-3ex} \frac{(\varphi(y)-\varphi(x)) }{|x-y|} (1\land |x-y|^p)\nu_\varepsilon(x-y) \,\d y= \frac{1}{2}K_{d,1}\nabla \varphi(x)\cdot e.
\intertext{Analogously one is able to show that}\label{eq:K1d-p}
&\lim_{\varepsilon \to 0} \il_{(y-x)\cdot e\leq 0} \hspace{-3ex} \frac{(\varphi(y)-\varphi(x)) }{|x-y|} (1\land |x-y|^p)\nu_\varepsilon(x-y) \,\d y = \frac{1}{2}K_{d,1}\nabla \varphi(x)\cdot e.
\end{align}
\noindent By substituting the two relations \eqref{eq:K1d} and \eqref{eq:K1d-p} in \eqref{eq:nonlocal-liminf}, using the dominate convergence theorem one readily ends up with the desired estimate.
\end{proof}

\smallskip

\begin{proof}[\textbf{Proof of Theorem \ref{thm:liminf}}]
	The estimate \eqref{eq:cont-weak-derivatve} holds true for all $\varphi\in C_c^\infty (\Omega)$, all $1\leq p<\infty$ and $e=e_i,~~i=1,\cdots,d$ so that $\nabla \varphi (x)\cdot e_i = \partial_{x_i}\varphi (x)$.

\smallskip 
\noindent \textbf{Case $1<p<\infty$:} In virtue of the density of $C_c^\infty (\Omega)$ in $L^{p'}(\Omega)$, it readily follows from \eqref{eq:cont-weak-derivatve} that for each $i=1,\cdots,d$ the mapping $\varphi\mapsto \int_{\Omega} u(x) \partial_{x_i}\varphi (x)\d x $
uniquely extends as a continuous linear form on $L^{p'}(\Omega)$. Since $1<p'<\infty$, the Riesz representation for Lebesgue spaces reveals that there exists a unique $g_i\in L^p(\Omega)$ and we set $\partial_{x_i} u= -g_i$, such that
	\begin{align*}
	\int_{\Omega} u(x) \partial_{x_i}\varphi (x)~\mathrm{d}x = \int_{\Omega} g_i(x) \varphi (x)~\mathrm{d}x=- \int_{\Omega} \partial_{x_i} u(x) \varphi (x)~\mathrm{d}x \quad\text{for all} \quad \varphi\in C_c^\infty(\Omega). 
	\end{align*}
	
\noindent 	In order words, $u\in W^{1,p}(\Omega)$.  Further, the $L^p$-duality and  \eqref{eq:cont-weak-derivatve} yields the estimate \eqref{eq:gradient-estimate} as follows
%	\vspace{-1ex}
\begin{align*}
\|\nabla u\|_{L^p(\Omega)} \leq \sqrt{d}\sum_{i=1}^d \|\partial_{x_i} u\|_{L^p(\Omega)}
&=\sqrt{d}\sum_{i=1}^d \sup_{\stackrel{\varphi\in C_c^\infty(\R^d)}{\|\varphi\|_{L^{p'}(\Omega)}=1}} \Big|\int_{\Omega} u(x) \, \nabla\varphi (x)\cdot e_i\d x\Big|\leq d^{2}\frac{A_p^{1/p}}{ K_{d,1}}.
\end{align*}
%\vspace{-1ex}

\noindent \textbf{Case $p=1$:} Let $\chi =(\chi_1,\chi_2,\cdots,\chi_d)\in C_c^\infty(\Omega, \R^d)$ such that $\|\chi\|_{L^\infty(\Omega,\R^d)}\leq 1$ and $e=e_i,\,\,i=1,2\cdots, d$. Since $\chi_i\in C_c^\infty(\Omega)$, the estimate \eqref{eq:cont-weak-derivatve} implies 
\begin{align*}
\Big|\int_{\Omega} u(x) \, \operatorname{div}\chi \d x\Big| 
&=\Big|\sum_{i=1}^d\int_{\Omega} u(x) \, \nabla\chi_i(x)\cdot e_i\d x\Big|\leq d\frac{A_1}{K_{d,1}}. 
\end{align*} 
\noindent Hence $u\in BV(\Omega)$ and we have $|u|_{BV(\Omega)}\leq d\frac{A_1}{K_{d,1}}$ which is the estimate \eqref{eq:gradient-estimate}.
\end{proof}

%\smallskip

\noindent The next result improves the estimate \eqref{eq:gradient-estimate}. 
\begin{theorem}\label{thm:liminf-BBM} Let $\Omega\subset \R^d$ be open. If $u \in L^p(\Omega)$ with $1<p<\infty$ or $u\in W^{1,1}(\Omega)$ for $p=1$ then
	\begin{align*}
 &	K_{d,p} \| \nabla u\|^p_{L^p(\Omega)} \leq	\liminf_{\varepsilon\to 0} \iil_{\Omega\Omega}|u(x) -u(y)|^p \nu_\varepsilon(x-y)\d x\d y=A_p.
 \end{align*}
\noindent Moreover if $p=1$ and $u\in L^1(\Omega)$ then we have 
 \begin{align*}
 &		K_{d,1} |u|_{BV(\Omega)}\leq	\liminf_{\varepsilon\to 0} \iil_{\Omega\Omega}|u(x) -u(y)| \nu_\varepsilon(x-y)\d x\d y= A_1.
	\end{align*} 
\end{theorem}

\begin{proof} 
%We propose two different proofs here. 
\noindent\textbf{First proof.} 
For $\delta >0$ small, set $\Omega_\delta =\{x\in \Omega: \operatorname{dist}(x,\partial\Omega)>\delta\}.$
	 Define the mollifier $\phi_\delta(x)= \frac{1}{\delta^d}\phi\left(\frac{x}{\delta}\right)$ with support in $B_\delta(0)$ where $\phi \in C_c^{\infty}(\R^d)$ is supported in $B_1(0)$, $\phi \geq 0$ and $ \int_{} \phi = 1$. 
	We assume that $u$ is extended by zero off $\Omega$ and let $u^\delta = u*\phi_\delta$ is the convolution product of $u $ and $\phi_\delta$. If $z\in\Omega_\delta $ and $|h|\le \delta$ then $z-h\in \Omega_\delta-h \subset \Omega$. A change of variables implies 
	\begin{align*}
	\iil_{\Omega\Omega}|u(x) -u(y)|^p \nu_\varepsilon(x-y)\d y\,\d x
	&\geq\hspace{-2ex} \iil_{\Omega_\delta-h\Omega_\delta-h}|u(x) -u(y)|^p \nu_\varepsilon(x-y)\d y\,\d x\\
	&=\iil_{\Omega_\delta\Omega_\delta} |u(x-h) -u(y-h)|^p \nu_\varepsilon(x-y)\d x\d y.
	\end{align*}
	Thus given that $\int_{} \phi_\delta \d h= 1$, integrating with respect to 
	$\phi_\delta(h)dh$, Jensen's inequality yields 
	\begin{align*}
	\iil_{\Omega\Omega}|u(x) -u(y)|^p \nu_\varepsilon(x-y)\d y\,\d x
	&\geq \int_{\R^d} \phi_\delta(h)\d h \iil_{\Omega_\delta\Omega_\delta} |
	u(x-h) -u(y-h)|^p \nu_\varepsilon(x-y)\d y\,\d x\\ 
	%&= \iil_{\Omega_\delta\Omega_\delta} \int_{\R^d} |u(x-h) -u(y-h)|^p \phi_\delta(h) \d h \,\nu_\varepsilon(x-y)\d x\d y\\ 
	& \geq \iil_{\Omega_\delta\Omega_\delta} \Big| \int_{\R^d} \big(u(x-h) -u(y-h)\big) \phi_\delta(h) \d h\Big|^p \nu_\varepsilon(x-y)\d x\d y\\
	&= \iil_{\Omega_\delta\Omega_\delta} |u*\phi_\delta (x) -u*\phi_\delta (y)|^p \nu_\varepsilon(x-y)\d x\d y.
	\end{align*}
	In other words, we have 
	\begin{align}\label{eq:molification-convex-Jessen}
	\iil_{\Omega_\delta\Omega_\delta} |u^\delta (x) -u^\delta (y)|^p \nu_\varepsilon(x-y)\d x\d y\leq \iil_{\Omega\Omega}|u(x) -u(y)|^p \nu_\varepsilon(x-y)\d x\d y.
	\end{align}
	Note that $u^\delta\in C^\infty(\R^d)$ and $\Omega_{\delta,j}= \Omega_\delta\cap B_j(0)$ has a compact closure for each $j\geq 1$.
	%and thus $u^\delta|_{\overline{\Omega_{\delta, j}} }$ belongs to $C^1_c(\overline{\Omega_{\delta, j}})$. 
	Then for each $j\geq 1$ the Lemma \ref{lem:BBM-regular} implies
	\begin{align*}
	K_{d,p}\int_{\Omega_{\delta,j} } |\nabla u^\delta (x)|^p \d x
	&=	\lim_{\varepsilon\to 0} \iil_{\Omega_{\delta,j}\Omega_{\delta,j}}|u(x) -u(y)|^p \nu_\varepsilon(x-y)\d x\d y\\
	&\leq 	\liminf_{\varepsilon\to 0} \iil_{\Omega\Omega}|u(x) -u(y)|^p \nu_\varepsilon(x-y)\d x\d y=A_p.
	\end{align*}
	Tending $j\to\infty$ in the latter we get
	\begin{align}\label{eq:limit-approx}
	K_{d,p}\int_{\Omega_\delta} |\nabla u^\delta (x)|^p \d x\leq A_p.
	\end{align}
	
\noindent \textbf{Case $1<p<\infty$: } 	The only interesting scenario occurs if $A_p<\infty$. In this case, Theorem \ref{thm:liminf} ensures that $u\in W^{1,p}(\Omega)$. Clearly we have $\nabla u^\delta= \nabla( u*\phi_\delta)= \nabla u*\phi_\delta $ and  $\| \phi_\delta* \nabla u - \nabla u \|_{L^p(\Omega)}\to 0$ as $\delta \to 0$.
	 The desired inequality follows by letting $\delta\to 0$ in \eqref{eq:limit-approx} since 
	\begin{align*}
	\Big| \| \nabla u \|_{L^p(\Omega)} -\|\nabla u *\phi_\delta \|_{L^p(\Omega_\delta)} \Big|\leq \| \nabla u *\phi_\delta- \nabla u \|_{L^p(\Omega)} + \Big(\int_{\Omega\setminus \Omega_\delta} |\nabla u (x)|^p\d x\Big)^{1/p}\xrightarrow{\delta\to 0} 0.
	\end{align*}
	
\noindent \textbf{Case $p=1$: } Again we only need to assume that $A_1<\infty$ so that by Theorem \ref{thm:liminf}, $u\in BV(\Omega)$. The relation \eqref{eq:limit-approx} implies that 
	\begin{align*}
	K_{d,1}\liminf\limits_{\delta\to 0}\int_{\Omega_\delta} |\nabla u^\delta (x)| \d x\leq A_1.
	\end{align*}
	
		\noindent Let $\chi\in C_c^\infty(\Omega, \R^d)$ such that $\|\chi\|_{L^\infty(\Omega,\R^d)}\leq 1$ and $\operatorname{supp} \chi\subset \Omega_\delta$ for $\delta>0$ small. We find that 
	\begin{align*}
	\Big| \int_{\Omega} u(x) \operatorname{div} \chi(x)\d x - \int_{\Omega_\delta} u^\delta (x) \operatorname{div} \chi(x)\d x\Big|&=\Big| \int_{\Omega_\delta} (u (x) -u*\phi_\delta (x)) \operatorname{div} \chi(x)\d x %+\int_{\Omega\setminus \Omega_\delta} \hspace{-3ex} u*\phi_\delta(x) \operatorname{div} \chi(x)\d x
 \Big|\\&\leq\|\operatorname{div}\chi\|_{L^\infty(\Omega,\R^d)}\|u*\phi_\delta-u\|_{L^1(\Omega)} %+\|\operatorname{div} \chi\|_{L^\infty(\Omega)} \|\phi\|_{L^1(\R^d)}\int_{\Omega\setminus \Omega_\delta} |u (x) |\d x
	\xrightarrow{\delta\to 0}0. 
	\end{align*}

	\noindent Thus, since $u$ is a distribution on $\Omega$ we get
	\begin{alignat*}{2}
	\int_{\Omega} u(x) \operatorname{div} \chi(x)\d x&=\lim_{\delta\to 0} \int_{\Omega_\delta} u^\delta (x) \operatorname{div} \chi(x)\d x\\
	&=-\lim_{\delta\to 0} \int_{\Omega_\delta} \nabla u^\delta (x) \cdot \chi(x)\d x 
	\leq \liminf_{\delta\to 0}\int_{\Omega_\delta} |\nabla u^\delta (x)| \d x.
	\end{alignat*}
\noindent 	This completes the proof since the above holds for arbitrarily chosen $\chi\in C_c^\infty(\Omega, \R^d)$ such that $\|\chi\|_{L^\infty(\Omega,\R^d)}\leq 1$, by definition of $|\cdot |_{BV(\Omega)}$ and the previous estimate we get
	\begin{align*}
	K_{d,1} |u|_{BV(\Omega)}\leq
	\liminf_{\delta\to 0}\int_{\Omega_\delta} |\nabla u^\delta (x)| \d x \leq A_1. 
	\end{align*}

\noindent \textbf{Second proof.} Here is an alternative. Since for all $\delta>0$, $\int_{B_\delta(0)}|h|^p\nu_\eps(h)\d h\to1$ as $\eps\to0$ (see the formula \eqref{eq:concentration-bis}), for each compact set $K\subset \Omega$ and $\eta>0$ inequality \eqref{eq:optimal-non-levy} implies
\begin{align*}%\label{eq:optimal-non-levy}
\begin{split}
\liminf_{\eps\to 0}\iil_{\Omega\Omega}|u(x) -u(y)|^p \nu_\varepsilon(x-y)\d y\,\d x&\geq \liminf_{\eps\to 0}\big(\int_{B_\delta(0)}|h|^p\nu_\eps(h)\d h\big)\big(K^{1/p}_{d,p}\|\nabla u\|_{L^p(K)}-\eta\big)^p\\
&=\big(K^{1/p}_{d,p}\|\nabla u\|_{L^p(K)}-\eta\big)^p. 
\end{split}
\end{align*}
Let $K_j=\overline{\Omega}_j\subset \Omega_{j+1}$ and $(\Omega_j)_j$ be an exhaustion of $\Omega$. Since the above inequality is true for every compact set $K=K_j\subset \Omega$ and every $\eta>0$ we conclude that 
\begin{align*}
\liminf_{\eps\to 0}\iil_{\Omega\Omega}|u(x) -u(y)|^p \nu_\varepsilon(x-y)\d y\,\d x\geq K_{d,p}\|\nabla u\|^p_{L^p(\Omega)}.
\end{align*}
The case $p=1$ and $u\in BV(\Omega)$, follows from  the approximation Theorem \ref{thm:meyer-serrin-bv}. 
\end{proof}

\noindent It is worth to mention that the convolution technique used in the first proof above was  first used in \cite{Brezis-const-function} when $\Omega=\R^d$ and also appears in \cite{Pon04}. The next theorem is a the counterpart of Theorem \ref{thm:liminf-BBM} and is a refinement version of Theorem \ref{thm:limpsup}.

\begin{theorem}\label{thm:limsup_BBM} Let $\Omega\subset \R^d$ be a $W^{1,p}$-extension domain and $u\in L^p(\Omega)$, $p>1$ then we have 
	\begin{align*}
	\limsup_{\varepsilon\to 0}\iil_{\Omega\Omega}|u(x)-u(y)|^p\nu_\varepsilon(x-y) \mathrm{d}y \d x\leq	K_{d,p} \| \nabla u\|^p_{L^p(\Omega)}. 
	\end{align*}
Moreover, for $p=1$, if $\Omega$ is a $BV$-extension domain and $u\in L^1(\Omega)$ then we have
\begin{align*}	\limsup_{\varepsilon\to 0}\iil_{\Omega\Omega}|u(x)-u(y)|\nu_\varepsilon(x-y) \mathrm{d}y \d x\leq	K_{d,1} 
	 |u|_{BV(\Omega)}. 
	\end{align*}
\end{theorem}

\smallskip

\begin{proof}
	The cases $\|\nabla u\|_{L^p(\Omega)} =\infty$ and $|u|_{BV(\Omega)}=\infty$ are trivial. Assume $u\in W^{1,p}(\Omega)$ and let $\overline{u}\in W^{1,p}(\R^d)$ be its extension to $\R^d$. Consider 
	$\Omega(\delta) = \Omega+ B_\delta(0) = \{x\in\R^d~:\operatorname{dist}(x, \Omega)<\delta \}$ be a neighborhood of $\Omega$ where $0<\delta<1$. We claim that for each $\varepsilon>0$, the following estimate holds 
	\begin{align}\label{eq:split-estimate}
	\iil_{\Omega\Omega}|u(x)-u(y)|^p\nu_\varepsilon(x-y) \mathrm{d}y \d x\leq K_{d,p}\int_{\Omega(\delta) }|\nabla \overline{u}(x)|^p\d x+ 2^p\|u\|^p_{L^p(\Omega)}\int_{|h|>\delta}\nu_\varepsilon(h)\,\d h.
	\end{align}

	\noindent Indeed, let $(u_n)_n$ be a sequence in $C_c^\infty(\R^d)$ converging to $\overline{u}$ in $W^{1,p}(\R^d)$.  For each $n\geq 1$, passing through the polar coordinates and using the identity \eqref{eq:rotation-invariant-constant} we find that
	\begin{align*}
	\iil_{ \Omega\times \Omega\cap \{|x-y|\leq \delta\}} \hspace{-2ex}|u_n(x)-u_n(y)|^p\nu_\varepsilon(x-y)\,\d y \,\d x
	&\leq \int_0^1\int_\Omega \int_{|h|\leq \delta} |\nabla u_n(x+th)\cdot h|^p\, \nu_\varepsilon(h)\,\d h \,\d x \d t\\
	&\leq \int_{|h|\leq \delta}\int_{\Omega(\delta)}|\nabla u_n(z)\cdot h|^p\,\d z ~\nu_\varepsilon(h)\,\d h\\
	&= \Big(\int_{\Omega(\delta)}\int_{\mathbb{S}^{d-1}} \hspace{-2ex}|\nabla u_n(z)\cdot w|^p\,\d\sigma_{d-1}(w)\Big)\Big(\int_{0}^\delta r^{p+d-1}\nu_\varepsilon(r)\,\d r\Big)\\
	&= K_{d,p} \Big(\int_{\Omega(\delta)}|\nabla u_n(z)|^p\,\d z \Big)\Big(\int_{|h|\leq \delta}(1\land |h|^p)\nu_\varepsilon(h)\,\d h\Big)\\
	&\leq K_{d,p} \int_{\Omega(\delta)}|\nabla u_n(z)|^p\,\d z. 
	\end{align*}
	\noindent Fatou's lemma implies 
	\begin{align*}
	\iil_{ \Omega\times \Omega\cap \{|x-y|\leq \delta\}} \hspace{-2ex}|u(x)-u(y)|^p\nu_\varepsilon(x-y)\,\d y \,\d x
	&\leq \liminf_{n\to\infty}\int_{\Omega} \int_{|x-y|\leq \delta} \hspace{-2ex}|u_n(x)-u_n(y)|^p\nu_\varepsilon(x-y)\,\d y \,\d x\\
	&\leq K_{d,p} \int_{\Omega(\delta)}\, |\nabla \overline{u}(x)|^p\d x.
	\end{align*}
The estimate \eqref{eq:split-estimate} clearly follows since we have
	\begin{align*}
	\int_{\Omega} \int_{\Omega\cap \{|x-y|> \delta\}} \hspace{-3ex} |u(x)-u(y)|^p\nu_\varepsilon(x-y)\,\d y \,\d x \leq 2^p \|u\|^p_{L^p(\Omega)}\int_{|h|>\delta}\nu_\varepsilon(h)\,\d h. 
	\end{align*}
	Letting $\varepsilon \to0$ the relation \eqref{eq:split-estimate} yields
	\begin{align*}
	\limsup_{\varepsilon\to 0} \iil_{\Omega\Omega}|u(x)-u(y)|^p\nu_\varepsilon(x-y) \d y \d x\leq K_{d,p}\int_{\Omega(\delta)}|\nabla u(x)|^p\d x
	\end{align*}
	Recalling that $\overline{u}\in W^{1,p}(\R^d)$, $u= \overline{u}\mid_\Omega$ and using \eqref{eq:lp-boundary-extension} the desired estimate follows 
	$$\int_{\Omega(\delta) }|\nabla \overline{u}(x)|^p\d x \xrightarrow{\delta\to 0} \int_{\overline{\Omega}}|\nabla \overline{u}(x)|^p\d x =\int_{\Omega}|\nabla u(x)|^p\d x.$$

\noindent If $p=1$ and $u\in BV(\Omega)$, let $\overline{u}\in BV(\R^d)$ be its extension to $\R^d$. By Theorem \ref{thm:meyer-serrin-bv} there is $(u_n)_n$ a sequence in $C^\infty(\R^d)\cap W^{1,1}(\R^d)$ which converges to $\overline{u}$ in $L^1(\R^d)$ and $\|\nabla u_n\|_{L^1(\R^d)}\xrightarrow{n \to \infty}| \overline{u}|_{BV(\R^d)}$. 

\noindent The estimate \eqref{eq:split-estimate} applied to $u_n$ and the Fatou's lemma yield 
\begin{align*}
\iil_{\Omega\Omega}|u(x)-u(y)|\nu_\varepsilon(x-y) \mathrm{d}y \d x
&\leq \liminf_{n\to\infty}\iil_{\Omega\Omega}|u_n(x)-u_n(y)|\nu_\varepsilon(x-y) \mathrm{d}y \d x \\
&\leq \lim_{n\to\infty} K_{d,1} \int_{\Omega(\delta)}|\nabla u_n(x)|\,\d x + 2\|u_n\|_{L^1(\Omega)}\int_{|h|>\delta}\nu_\varepsilon(h)\,\d h.\\
&= K_{d,1} |\overline{u}|_{BV(\Omega(\delta))} + 2\|u\|_{L^1(\Omega)}\int_{|h|>\delta}\nu_\varepsilon(h)\,\d h.
\end{align*}
Correspondingly, we also get the estimate  
\begin{align}\label{eq:split-estimate-BV}
\iil_{\Omega\Omega}|u(x)-u(y)|\nu_\varepsilon(x-y) \mathrm{d}y \d x
\leq K_{d,1} |\overline{u}|_{BV(\Omega(\delta))} + 2\|u\|_{L^1(\Omega)}\int_{|h|>\delta}\nu_\varepsilon(h)\,\d h.
\end{align}
\noindent Therefore, letting $\varepsilon\to 0$ implies that
\begin{align*}
\limsup_{\varepsilon\to 0} \iil_{\Omega\Omega}|u(x)-u(y)|^p\nu_\varepsilon(x-y) \d y \d x\leq K_{d,1}
| \overline{u}|_{BV(\Omega(\delta))}. 
\end{align*}
\noindent Recalling that $\overline{u}\in BV(\R^d) $, $u= \overline{u}\mid_\Omega$ and  $\partial\Omega$ satisfies \eqref{eq:bv-boundary-extension}, i.e., $|\nabla \overline{u}|(\partial\Omega)=0$ we have 
$$ | \overline{u}|_{BV(\Omega(\delta))}\xrightarrow{\delta\to 0} |\overline{u}|_{BV(\overline{\Omega})}=|u|_{BV(\Omega)}.$$
\end{proof}

\noindent The following result involves the collapse across the boundary $\partial\Omega$. 
%is a consequence of Theorem \ref{thm:liminf-BBM}.  
\begin{theorem}\label{thm:collap-bdary}
Assume $\Omega\subset \R^d$ is  open  then for any $u\in W^{1,p}(\R^d)$ we have 
\begin{align*}
\limsup_{\eps\to 0}2\iil_{\Omega \Omega^c} |u(x)-u(y)|^p\nu_\eps(x-y)\d y\,\d x
\leq  K_{d,p}\int_{\partial \Omega}|\nabla u(x)|^p\d x,\\
%Use the liminf the first and limsup for the second on $\partial \overline{\Omega}\subset \partial \Omega$
\liminf_{\eps\to 0}2\iil_{\Omega \Omega^c} |u(x)-u(y)|^p\nu_\eps(x-y)\d y\,\d x\geq K_{d,p}\int_{\partial \overline{\Omega}}|\nabla u(x)|^p\d x.
\end{align*}
The same holds for $p=1$by replacing  $W^{1,1}(\R^d)$ with $BV(\R^d)$. 
\end{theorem}

\begin{proof}
We only prove for $ u\in W^{1,p}(\R^d)$, the case $u\in BV(\R^d)$ is analogous. The sets $\Omega$ and $U_\delta=\{x\in \R^d: \dist(x,\Omega)>\delta \}$, $\delta>0$ are open. By Theorem \ref{thm:liminf-BBM}, we get  %\linebreak[-2]
\begin{align*}
&K_{d,p}\int_{\Omega}|\nabla u(x)|^p\d x
\leq \liminf_{\eps\to 0}  \iil_{\Omega \Omega} |u(x)-u(y)|^p\nu_\eps(x-y)\d y\,\d x,\\
%\leq K_{d,p} \|\nabla u\|^p_{L^p(\R^d)},\\
& K_{d,p}\int_{U_\delta}|\nabla u(x)|^p\d x 
\leq \liminf_{\eps\to 0} \iil_{U_\delta\, U_\delta} |u(x)-u(y)|^p\nu_\eps(x-y)\d y\,\d x. 
\end{align*}
Since  $U_\delta\subset \Omega^c$ and  $\mathds{1}_{U_\delta}(x)\to  \mathds{1}_{\R^d\setminus\overline{\Omega}}(x)$,  for all $x\in \R^d$ as $\delta\to 0$, it follows that
%the convergence dominated theorem implies 
\begin{align*}
K_{d,p}\int_{\R^d\setminus\overline{\Omega}}\hspace{-0.4ex} |\nabla u(x)|^p\d x
\leq \liminf_{\eps\to 0}  \iil_{\Omega^c \Omega^c} |u(x)-u(y)|^p\nu_\eps(x-y)\d y\,\d x. 
%\leq K_{d,p} \|\nabla u\|^p_{L^p(\R^d)}.  
\end{align*}
Accordingly, together with  \eqref{eq:lim-full-form}, we deduce the desired result as follows  
\begin{align*}
\limsup_{\eps\to 0}2\iil_{\Omega\Omega^c} &|u(x)-u(y)|^p\nu_\eps(x-y)\d y\,\d x\\
&=\limsup_{\eps\to0}\Big(\hspace{-0.5ex}\iil_{\R^d\R^d}-\iil_{\Omega\Omega} -\iil_{\Omega^c\Omega^c}\hspace{-0.5ex}\Big)|u(x)-u(y)|^p\nu_\eps(x-y)\d y\,\d x\\
%&=\limsup_{\eps\to 0}\big(\mathcal{E}^\eps_{\R^d}(u,u) -\mathcal{E}^\eps_{\Omega}(u,u)-\mathcal{E}^\eps_{\Omega^c}(u,u)\Big)\\
%&\leq \limsup_{\eps\to 0}\mathcal{E}^\eps_{\R^d}(u,u) -\liminf_{\eps\to 0}\mathcal{E}^\eps_{\Omega}(u,u)-\liminf_{\eps\to 0}\mathcal{E}^\eps_{\Omega^c}(u,u)\\
%&\leq K_{d,p} \Big(\int_{\R^d}-\int_{\Omega}-\int_{\Omega^c}\Big)|\nabla u(x)|^p\d x
&\leq K_{d,p} \Big(\|\nabla u\|^p_{L^p(\R^d)}-\|\nabla u\|^p_{L^p(\Omega)}-\|\nabla u\|^p_{L^p(\R^d\setminus\overline{\Omega})}\Big)= K_{d,p}\int_{\partial \Omega}|\nabla u(x)|^p\d x. 
\end{align*}
The reverse inequality follows analogously, since by exploiting \eqref{eq:split-estimate} (or \eqref{eq:split-estimate-BV}) one easily gets
\begin{align*}
&K_{d,p}\int_{\overline{\Omega}}|\nabla u(x)|^p\d x
\geq \limsup_{\eps\to 0}  \iil_{\Omega \Omega} |u(x)-u(y)|^p\nu_\eps(x-y)\d y\,\d x,\\
& K_{d,p}\int_{\Omega^c}|\nabla u(x)|^p\d x 
\geq \limsup_{\eps\to 0} \iil_{\Omega^c\Omega^c} |u(x)-u(y)|^p\nu_\eps(x-y)\d y\,\d x. 
\end{align*}
\end{proof}

\smallskip 

\noindent Next we establish a pointwise and $L^1(\Omega)$ convergence when $u$ is a sufficiently smooth function. 

\begin{lemma}\label{lem:BBM-regular}
	Let $\Omega\subset \R^d$ be open and $u\in C^1_c(\R^d)$. The following convergence occurs in both pointwise and $L^1(\Omega)$ sense:
	\begin{align*}
	\lim_{\varepsilon\to 0}\int_{\Omega}|u(x)-u(y)|^p\nu_\varepsilon(x-y) \mathrm{d}y = K_{d,p} | \nabla u(x)|^p.
	\end{align*}
\end{lemma}

\smallskip 

\begin{proof}
\textbf{First proof of Lemma \ref{lem:BBM-regular}}. Let $\sigma>0$ be sufficiently small. By assumption $\nabla u$ is uniformly continuous and hence one can find $0<\eta=\eta(\sigma) <1$ such that if $|x-y|<\eta$ then 
	\begin{align}\label{eq:continuity-estimate}
	%\frac{1}{2} |\nabla u(x)|\leq 	|\nabla u(y)| \leq 2 |\nabla u(x)|\qquad\hbox{and}\qquad 
	|\nabla u(y)-\nabla u(x)| \leq \sigma.
	\end{align}

\noindent Let $\eta_x = \min (\eta, \delta_x)$ with $\delta_x= \operatorname{dist}(x, \partial \Omega)$ so that $B(x, \eta_x)\subset \Omega$ for all $x\in \Omega$. Consider the mapping $F:\Omega\times (0,1)\to \mathbb{R}$ with
	\begin{align*}
	F(x,\varepsilon):= \int\limits_{\Omega\cap \{|x-y|\leq \eta_x\}} \hspace{-2ex}|u(x)-u(y)|^p \nu_\varepsilon(x-y)\mathrm{d}y= \int_{|h|\leq \eta_x} |u(x)-u(x+h)|^p \nu_\varepsilon(h)\,\d h. 
	\end{align*} 
	In virtue of the fundamental theorem of calculus, we have 
	\begin{align*}
	F(x,\varepsilon)&=\int_{|h|\leq \eta_x} \Big|\int_{0}^{1}\nabla u(x+th)\cdot h\mathrm{d}t\Big|^p \nu_\varepsilon(h)\,\d h =\int_{|h|\leq \eta_x} |\nabla u(x)\cdot h|^p\, \nu_\varepsilon(h)\,\d h + R(x,\varepsilon), 
	\intertext{with the remainder}
	R(x,\varepsilon) &= \int_{|h|\leq \eta_x} \left(\Big|\int_{0}^{1}\nabla u(x+th)\cdot h\mathrm{d}t\Big|^p- \Big|\int_{0}^{1}\nabla u(x)\cdot h\mathrm{d}t\Big|^p \right)\nu_\varepsilon(h)\,\d h.
	\end{align*}
	The mapping $s \mapsto G_p(s)=|s|^p$ belongs to $C^1(\R^d\setminus \{0\})$ and $G'_p(s) = pG_p(s) s^{-1}$. Thus, we have 
	\begin{align*}
	G_p(b)-G_p(a)	= (b-a) \int_{0}^{1} G'_p(a+s(b-a))\,\d s. 
	\end{align*}
Set $a= \nabla u (x)\cdot h$ and $b= \int_0^1\nabla u (x+th)\cdot h \,\d t$ so that the relation \eqref{eq:continuity-estimate} yields
	\begin{align*}
	|G_p(b)-G_p(a)|	&\leq p|b-a| \int_{0}^{1} |(1-s)a+sb|^{p-1}\,\d s\\
	&\leq p\|\nabla u\|^{p-1}_{L^\infty(\R^d)} |h|^{p-1}\int_0^1 |\nabla u(x+th)-\nabla u(x)||h|\,\d t\\
	&\leq p \sigma\|\nabla u\|^{p-1}_{L^\infty(\R^d)} |h|^{p}. 
	\end{align*}
Integrating both sides with respect to $\nu_\eps(h)\d h$, implies that 
	%\begin{relsize}{-1}
\begin{align*}
|R(x,\varepsilon)|
%&:= \Big| \int_{|h|\leq \eta_x} \left(\Big|\int_{0}^{1}\nabla u(x+th)\cdot h\mathrm{d}t\Big|^p- \Big|\int_{0}^{1}\nabla u(x)\cdot h\mathrm{d}t\Big|^p \right)\nu_\varepsilon(h)\,\d h\Big|\\
%
&\leq p \sigma\|\nabla u\|^{p-1}_{L^\infty(\R^d)} \int_{|h|\leq \eta_x} |h|^p\nu_\varepsilon(h)\,\d h.
%= p \sigma\|\nabla u\|^{p-1}_{L^\infty(\R^d)} \il_{|h|\leq \eta_x} (1\land |h|^p)\nu_\eps(h)\,\d h
%\xrightarrow{\varepsilon, \sigma \to 0} 0.
\end{align*}
\noindent Since $\int_{|h|\leq \eta_x}  |h|^p\nu_\eps(h)\,\d h\to 1$ as $\eps \to 0$, by the formula \eqref{eq:concentration-bis}, letting $\eps\to 0$ and $\sigma\to0$ successively 
yields $R(x,\eps)\to 0$. 
\noindent Whereas, using polar coordinates,  the relation \eqref{eq:rotation-invariant-constant} and the Remark \ref{rem:asymp-nu} gives 
	
	\begin{align*}
	\int_{|h|\leq \eta_x} |\nabla u(x)\cdot h|^p \, \nu_\varepsilon(h)\,\d h
	&= \int_{0}^{\eta_x} r^{d+p-1} \mathrm{d}r \int_{\mathbb{S}^{d-1}}|\nabla u(x)\cdot w|^p\mathrm{d}\sigma_{d-1}(w)\\
	&=K_{d,p} |\nabla u(x)|^p \int_{|h|\leq \eta_x} |h|^p\nu_\varepsilon(h)\,\d h \xrightarrow{\varepsilon \to 0} K_{d,p} |\nabla u(x)|^p.
	\end{align*}
Therefore, we have $F(x, \varepsilon) \xrightarrow{\eps\to0}K_{d,p} |\nabla u(x)|^p.$	Furthermore, a close look to our reasoning reveals that we have subsequently shown that 
	\begin{align}\label{eq:pointwise-close-limit}
	\lim_{\varepsilon\to 0}\il_{\Omega\cap \{|x-y|\leq \delta\}}\hspace{-2ex} |u(x)-u(y)|^p\nu_\varepsilon(x-y)\,\d y = K_{d,p} |\nabla u(x)|^p,\quad\text{for all $\delta>0$}.
	\end{align}
	This is due to the fact that, for all $\delta >0$  we have 
	\begin{align}\label{eq:limit-point-far}
	\int\limits_{ \Omega\cap \{|x-y|>\delta \}} \hspace{-3ex}|u(x)-u(y)|^p \nu_\varepsilon(x-y) \, \mathrm{d}y 
	\leq 2^p\|u\|^p_{L^{\infty}(\R^d)} \int_{|h|>\delta} \nu_\varepsilon(h) \mathrm{d}h
	\xrightarrow{\varepsilon \to 0} 0.
	\end{align}
	Hence we have the pointwise convergence as claimed, i.e., for all $x\in\Omega$ we have 
	\begin{align}\label{eq:pointwise-conv}
	\lim_{\varepsilon\to 0}\int_{\Omega} |u(x)-u(y)|^p \nu_\varepsilon(x-y) \d y\,\d x = K_{d,p} |\nabla u(x)|^p. 
	\end{align}
	To proceed with the convergence in $L^1(\Omega)$, according to the Sch\'eff\'e lemma \cite[p.55]{Wil91}, it suffices to show the convergence of $L^1(\Omega)$-norm. Choosing $R\geq1$ such that $\supp u\subset B_R(0)$, we write
	\begin{align*}
	\iil_{\Omega\Omega}|u(x)-u(y)|^p\nu_\varepsilon(x-y)\,\d y \,\d x
	&= \iil_{\Omega\times \Omega\cap 
		\{|x-y|\leq R\}} |u(x)-u(y)|^p\nu_\varepsilon(x-y)\,\d y \,\d x\\
	&+ \iil_{\Omega\times \Omega\cap 
		\{|x-y|>R\}} \hspace{-2ex} |u(x)-u(y)|^p\nu_\varepsilon(x-y)\,\d y \,\d x.
	\end{align*}
Since $|u(x)-u(x+h)|^p\leq 2^p(1\land |h|^p)\|u\|^p_{W^{1,\infty}(\R^d)}$ and $\int_{\R^d} (1\land |h|^p)\nu_\eps(h) \d h =1$ one gets
\begin{align*}
H_\eps(x)=	\il_{ \Omega\cap 
	\{|x-y|\leq R\}} \hspace{-2ex}|u(x)-u(y)|^p\nu_\varepsilon(x-y)\,\d y 
&\leq \int_{ |h|\leq R}\hspace{-1ex} |u(x)-u(x+h)|^p\nu_\varepsilon(h)\d h
%\\&\leq 2^p\|u\|^p_{W^{1,\infty}(\R^d)} \int_{ |h|\leq R}(1\land |h|^p) \nu_\varepsilon(h)\,\d h
\leq 2^p\|u\|^p_{W^{1,\infty}(\R^d)}.
\end{align*}

\noindent Noting that $\supp H_\eps \subset B_{2R}(0)$, one finds that $H_\eps(x)\leq 2^p\|u\|^p_{W^{1,\infty}(\R^d)}\mathds{1}_{B_{2R} (0)}$ with $\mathds{1}_{B_{2R} (0)}\in L^1(\Omega)$, the pointwise limit in \eqref{eq:pointwise-close-limit} and the dominated convergence theorem imply that 
\begin{align*}
\lim_{\varepsilon \to 0}\iil_{ \Omega\times \Omega \cap 
	\{|x-y|\leq R\}} \hspace{-2ex}|u(x)-u(y)|^p\nu_\varepsilon(x-y)\,\d y \,\d x=K_{d,p} \int_\Omega |\nabla u(x)|^p \,\d x.
\end{align*}
We thus obtain the following convergence of $L^1(\Omega)$-norm as expected 
\begin{align*}
\lim_{\varepsilon \to 0}\iil_{ \Omega \Omega} |u(x)-u(y)|^p\nu_\varepsilon(x-y)\,\d y \,\d x= K_{d,p}\int_\Omega |\nabla u(x)|^p \,\d x,
\end{align*}
since, by assumption on $\nu_\eps$, one has
\begin{align*}
\iint\limits_{\Omega \times \Omega\cap \{|x-y|>R \}} \hspace{-3ex}|u(x)-u(y)|^p \nu_\varepsilon(x-y) \d y\,\d x 
\leq 2^p\|u\|^p_{L^{p}(\Omega)} \int_{|h|>R} \nu_\varepsilon(h) \mathrm{d}h
\xrightarrow{\varepsilon \to 0} 0. 
\end{align*}
%\end{proof}
%\smallskip
%\begin{proof}[
\noindent \textbf{Second proof of Lemma \ref{lem:BBM-regular} when $u\in C_c^2(\R^d)$}. 
Note that if we put $G_p(s)=|s|^p$ then $G_p\in C^2(\R^d\setminus \{0\})$. The Taylor formula implies 
\begin{align*}
&u(y)-u(x)= \nabla u(x)\cdot (y-x) + O(|x-y|^2),\quad x,y\in \R^d, \\
&	G_p(b)- G_p(a) = G'_p(a)(b-a) +O(b-a)^2,\quad a,b\in \R\setminus\{0\}.
\end{align*}
Hence for almost all $x, y\in \R^d$, we have 
\begin{align*}
	|u(y)-u(x)|^p =G_p( \nabla u(x)\cdot (y-x) + O(|y-x|^2))= |\nabla u(x)\cdot (y-x)|^p + O(|y-x|^{p+1}).
\end{align*}
Set $\delta_x= \operatorname{dist}(x,\partial\Omega)$. Passing through polar coordinates and using the relation \eqref{eq:rotation-invariant-constant} yields
\begin{align*}
\int_{ B(x,\delta_x)}\hspace{-1ex} |u(x)-u(y)|^p&\nu_\varepsilon(x-y) \d y 
= \int_{|h|\leq \delta_x} \Big|\nabla u(x)\cdot h\Big |^p \nu_\varepsilon(h) \d h
+ O\Big( \int_{|h|\leq \delta_x} |h|^{p+1}\nu_\varepsilon(h) \d h\Big)\\
&= \int_{\mathbb{S}^{d-1}} \Big|\nabla u(x)\cdot w \Big|^p\d\sigma_{d-1} (w) \int_{0}^{\delta_x} r^{d-1}\nu_\varepsilon(r) \mathrm{d}r+ O\Big( \hspace{-1ex}\int_{|h|\leq \delta_x} |h|^{p+1} \nu_\varepsilon(h) \mathrm{d}h \Big)\\
&= K_{d,p}\left|\nabla u(x)\right|^p \int_{|h|\leq \delta_x} \nu_\varepsilon(h) \mathrm{d}h 
+ O\Big( \int_{|h|\leq \delta_x} |h|^{p+1} \nu_\varepsilon(h) \mathrm{d}h\Big).
\end{align*}

\noindent Therefore, letting $\varepsilon\to0 $ in the latter expression and taking into account Remark \ref{rem:asymp-nu} gives 
\begin{align*}
&\lim_{\varepsilon\to 0} \int_{ B(x,\delta_x)} |u(x)-u(y)|^p\nu_\varepsilon(x-y) \mathrm{d}y 
= K_{d,p}|\nabla u(x)|^p. 
\end{align*}
\noindent The pointwise convergence \eqref{eq:pointwise-conv} readily follows. Since on the other side, we have 

\begin{align*}
\int_{\Omega\setminus B(x,\delta_x)} \hspace{-1ex} |u(x)-u(y)|^p\nu_\varepsilon(x-y) \mathrm{d}y \leq 2^p\| u\|^p_{L^\infty(\Omega)} \int_{|h|\geq \delta_x} \hspace{-1ex} \nu_\varepsilon(h) \d h\xrightarrow{\eps\to 0}0.
\end{align*}
 Thus the remaining details follow by proceeding  as in the previous proof. 
\end{proof}

\smallskip 

\noindent We are now in position to prove Theorem \ref{thm:BBM-limit-result}. 
\begin{proof}[\textbf{Proof of Theorem \ref{thm:BBM-limit-result}}]
Assume $A_p=\infty$ then by Theorem \ref{thm:liminf} we have $\|\nabla u\|_{L^p(\Omega)}= \infty$ for $1<p<\infty$ and $|u|_{BV(\Omega)}=\infty$ for $p=1$. In either case the relation \eqref{eq:w1p-BBM-limit} or 
\eqref{eq:bv-BBM-limit} is verified. The interesting situation is when $A_p<\infty$, i.e., by Theorem \ref{thm:liminf}, $u\in W^{1,p}(\Omega)$ if $1<p<\infty$ and $u\in BV(\Omega)$ if $p=1$. We provide two alternative proofs. As first alternative, the result immediately follows by combining Theorem \ref{thm:liminf-BBM} and Theorem \ref{thm:limsup_BBM}. For the second alternative, consider $1<p<\infty$ or $u\in W^{1,1}(\Omega)$. 
By Lemma \ref{lem:boundedness-limsup} there is $C>0$ independent of $\varepsilon$ such that for $u, v\in W^{1,p}(\Omega)$,
\begin{align*}
	&\big|\|U_\varepsilon\|_{L^p(\Omega\times \Omega)}-\|V_\varepsilon\|_{L^p(\Omega\times \Omega)} \big| \leq\|U_\varepsilon-V_\varepsilon\|_{L^p(\Omega \times \Omega)}\leq C \|u-v\|_{W^{1,p}(\Omega)},\\
	\intertext{where we deifine}
&	U_\varepsilon(x,y) = |u(x)-u(y)|\nu_\varepsilon^{1/p}(x-y)\quad \text{and}\quad V_\varepsilon(x,y) = |v(x)-v(y)|\nu_\varepsilon^{1/p}(x-y)\,.
	\end{align*}
\noindent 	Therefore, it suffices to establish the result for $u$ in a dense subset of $W^{1,p}(\Omega)$. Note that $C_c^\infty(\R^d )$ is dense in $W^{1,p}(\Omega)$ since $\Omega$ is a $W^{1,p}$-extension domain. We conclude by using Lemma \ref{lem:BBM-regular}. 
\end{proof}

\noindent As consequence of Theorem \ref{thm:BBM-limit-result} we have the following concrete examples. 
\begin{corollary}\label{cor:BBM-fractional}
Assume $\Omega\subset \R^d$ is an extension domain and  $u\in L^p(\Omega)$. 
If we abuse the notation $\|\nabla u\|_{L^1(\Omega)} =|u|_{BV(\Omega)}$ for $p=1$, then there holds  
\begin{align*}
&\lim_{s\to 1} (1-s)\iint\limits_{\Omega\Omega}\frac{|u(x)-u(y)|^p}{|x-y|^{d+sp}}\mathrm{d}y\mathrm{d}x =\frac{|\mathbb{S}^{d-1}|}{p}K_{d,p} \|\nabla u\|^p_{L^p(\Omega)}, 
\\
&\lim_{\varepsilon \to 0} \varepsilon^{-d-p}\iint\limits_{\Omega\times\Omega\cap \{|x-y|<\varepsilon\} }|u(x)-u(y)|^p \d y\d x = \frac{|\mathbb{S}^{d-1}|}{d+p}K_{d,p}\|\nabla u\|^p_{L^p(\Omega)},\\
&\lim_{\varepsilon \to 0} \varepsilon^{-d}\iint\limits_{ \Omega\times\Omega\cap \{|x-y|<\varepsilon\} }\frac{|u(x)-u(y)|^p}{|x-y|^p} \d y\d x = \frac{|\mathbb{S}^{d-1}|}{d}K_{d,p}\|\nabla u\|^p_{L^p(\Omega)},\\
& 	
\lim_{\varepsilon \to 0} \frac{1}{|\log \varepsilon|}\iint\limits_{\Omega\times\Omega\cap \{|x-y|>\varepsilon\}} \frac{|u(x)-u(y)|^p}{|x-y|^{d+p}}\mathrm{d}y\mathrm{d}x = |\mathbb{S}^{d-1}|K_{d,p} \|\nabla u\|^p_{L^p(\Omega)}.
\end{align*}
\end{corollary} 

\begin{proof}
For the first relation, take $\nu_\varepsilon(h) = a_{\varepsilon, d,p} |h|^{-d-p(1-\varepsilon)}$ with $ a_{\varepsilon, d,p} = \tfrac{p\varepsilon(1-\varepsilon)}{|\mathbb{S}^{d-1}|}$.
For the second and third  take 
$\nu_\eps(h) =\frac{d+\beta}{|\mathbb{S}^{d-1}|}\varepsilon^{-d-\beta}\mathds{1}_{B_\varepsilon}(h)$ $\beta\in\{0,p\}$. For the last one, fixed $\varepsilon_0\geq 1$, take $\nu_\varepsilon(h) = \tfrac{ b_\varepsilon}{|\mathbb{S}^{d-1}||\log\varepsilon|}\mathds{1}_{B_{\varepsilon_0}\setminus B_\varepsilon}(h)$, $b_\varepsilon = \tfrac{p|\log\varepsilon|}{ (1-\varepsilon_0^{-p})+ p|\log\varepsilon|}$, where one notes that $b_\eps\to1$ as $\eps\to0$.
\end{proof}

\smallskip

\begin{proof}[\textbf{Proof of Theorem \ref{thm:weak-convergence}}]
Let $E\subset \Omega$ be compact with a nonempty interior. Consider the open set $E(\delta) = E+ B_\delta(0)\subset \Omega$ where $0<\delta< 1\land \operatorname{dist}(\partial\Omega, E)$ so that $\int_{|h|>\delta} \nu_\varepsilon(h)\,\d h\leq 1$.  Denote $ \d |\nabla u|^p(x)= |\nabla u(x)|^p\d x$, $u\in W^{1,p}(\Omega)$. Using \eqref{eq:split-estimate} and \eqref{eq:split-estimate-BV} with $\Omega$ replaced by $E$ imply 
\begin{align}\label{eq:split-estimate1}
	\int_E \mu_\varepsilon(x)\,\d x
	&\leq K_{d,p}\int_{E(\delta)} \hspace{-2ex}\d |\nabla u|^p(x)+ 2^p \|u\|^p_{L^p(\Omega)}\int_{|h|>\delta} \nu_\varepsilon(h)\,\d h.
	\end{align}

\noindent Hence, since $\int_{|h|>\delta} \nu_\varepsilon(h)\,\d h\leq 1$, the family of functions $(\mu_\varepsilon)_\varepsilon$ is bounded in $L^1(E)$. In virtue of the weak compactness of $L^1(E)$, (see \cite[p.116]{Bre10}) we may assume that $(\mu_\varepsilon)_\varepsilon$ converges in the weak-* sense to a Radon measure $\mu_E$, i.e., $\langle\mu_\eps-\mu_E, \varphi\rangle\xrightarrow{\eps\to 0}0$ for all $\varphi\in C(E)$ otherwise, one may pick a converging subsequence. For a suitable $(\Omega_j)_{j\in \mathbb{N}}$ exhaustion of $\Omega$, i.e., $\Omega_j's$ are open, each $K_j=\overline{\Omega}_j$ is compact, $K_j=\overline{\Omega}_j\subset \Omega_{j+1}$ and $\Omega=\bigcup_{j\in \mathbb{N}}\Omega_j$, it is sufficient to let $\mu= \mu_{K_j}= K_{d,p}|\nabla u|^p$ on $K_j$. We aim to show that $\mu= K_{d,p}|\nabla u|^p$. Noticing $\mu$ and $K_{d,p}|\nabla u|^p$ are Radon measures it sufficient to show that both measures coincide on compact sets, i.e., we have to show that $\mu_E(E) = K_{d,p}\int_{E} \,\d |\nabla u|^p(x). $ 
On the one hand,  since $\mu_\varepsilon(E)\to \mu(E)$ and $\int_{|h|>\delta}\nu_\varepsilon(h)\,\d h\to0$ as $\eps \to 0, $ the fact that $u\in W^{1,p}(\Omega)$ or $u\in BV(\Omega)$ enables us to successively let $\varepsilon\to 0$ and $\delta\to 0$ in \eqref{eq:split-estimate1} which amounts to the following
\begin{align*}
\int_E \d\mu_E(x)\,\leq K_{d,p}\int_{E} \,\d |\nabla u|^p(x). 
\end{align*}

\noindent On other hand, since $E$ has a nonempty interior, Theorem \ref{thm:liminf-BBM} implies
\begin{align*}
K_{d,p}\int_{E} \,\d |\nabla u|^p(x) &\leq	\liminf_{\varepsilon\to 0} \iil_{EE}|u(x) -u(y)|^p \nu_\varepsilon(x-y)\d y\d x
\leq \lim_{\varepsilon\to 0} \int_{E}\mu_\eps(x)\d x = \int_{E}\d\mu_E(x). 
\end{align*}

\noindent Finally $\mu(E)=\mu_E(E)= K_{d,p}\int_{E} \,\d |\nabla u|^p(x)$. Whence we get $\d \mu =K_{d,p} \d |\nabla u|^p$ as claimed.
\end{proof}
\noindent A consequence of Theorem \ref{thm:weak-convergence} is given by the following analog result.

\begin{corollary}\label{cor:weak-convergence}
	Let $\Omega\subset \R^d$ be open. Let $u\in W^{1,p}(\R^d)$ and define the Radon measures 
	\begin{align*}
	\d\widetilde{\mu}_{\varepsilon}(x) = \int_{\R^d}|u(x)-u(y)|^p\nu_\varepsilon(x-y) \mathrm{d}y\d x.
	\end{align*}
	The sequence $(\widetilde{\mu}_{\varepsilon})_\eps$
	converges weakly on $\Omega$ to the Radon measure $\d \mu(x) =K_{d,p}|\nabla u(x)|^p\d x$, 
	i.e. $ \widetilde{\mu}_\varepsilon (E)\xrightarrow{\eps\to 0}\mu(E)$ for every compact set $E\subset \Omega$. If $u\in BV(\Omega)$, $p=1$,  then $\d \mu(x) =K_{d,1}\d |\nabla u|(x)$.
\end{corollary}

\smallskip

\begin{proof}
	Let $E\subset \Omega$ be a compact set so that $\delta>0$ with $\delta= \dist(E,\Omega^c)>0$. Thus, we have 
	\begin{align*}
\widehat{\mu}_{\varepsilon}(E):=	\iil_{E\Omega^c} |u(x)-u(y)|^p\nu_\eps(x-y)\d y\d x
	%&\leq 2^{p-1}\iil_{\Omega\Omega^c} (|u(x)|^p+  |u(y)|^p)\mathds{1}_{B^c_\delta} (x-y)\nu_\eps(x-y)\d y\d x\\ &=
	&\leq2^{p-1}\|u\|^p_{L^p(\R^d)}\int_{ |h|>\delta}\nu_\eps(h)\d h\xrightarrow{\eps\to0}0. 
	\end{align*}
Form this and  Theorem \ref{thm:weak-convergence} we get $\widetilde{\mu}_{\varepsilon}(E)= \mu_{\varepsilon}(E)+\widehat{\mu}_{\varepsilon}(E) \xrightarrow{\eps\to0}\mu(E).$
\end{proof}

\noindent Next, we sate without proof the asymptotically compactness involving the case where the function $u$ also varies. The full proof can be found in \cite[Theorem 5.40]{Fog20} and \cite{Pon04}.

\begin{theorem}
	Assume $\Omega\subset \R^d$ is open, bounded and Lipschitz. Let the family $(u_\eps)_\eps$ such that 
	\begin{align*}
	\sup_{\eps>0} \Big(\|u_\eps\|^p_{L^p(\Omega)}+ \iil_{ \Omega \Omega} |u_\eps(x)-u_\eps(y)|^p\nu_\eps(x-y)\d y \d x\Big)<\infty.
	\end{align*}
	\noindent There exists a subsequence $(\eps_n)_n$ with $\eps_n\to0^+$ as $n\to \infty$ such that $(u_{\eps_n})_n$ converges in $L^p(\Omega)$ to a function $u\in L^p(\Omega)$. Moreover, $u\in W^{1,p}(\Omega)$ if $1<p<\infty$ or $u\in BV(\Omega)$ if $p=1$. 
\end{theorem}

%\smallskip
			
\begin{counterexample}\label{Ex:counterexample-extension}We consider the fractional kernel $\nu_\varepsilon(h) = a_{\varepsilon, d,p} |h|^{-d-(1-\eps)p}$, $p\geq1$, where $a_{\varepsilon, d,p} = \frac{p\varepsilon(1-\varepsilon)}{|\mathbb{S}^{d-1}|}.$ We put $s= 1-\varepsilon>0$ and consider the nonlocal seminorm
\begin{align*}
|u|^p_{W^{s,p}(\Omega)}=|u|^p_{W^{p}_{\nu_\eps}(\Omega)}
= \iil_{\Omega\Omega}|u(x)- u(y)|^p\nu_\eps(x-y)\d y\d x
=\frac{ps(1-s)}{|\mathbb{S}^{d-1}|} \iil_{\Omega\Omega} \frac{|u(x)- u(y)|^p}{|x-y|^{d+sp}}\d y \d x.
\end{align*}

\textbf{Case $d=1$.} For an illustrative purpose we start with the case $d=1$. Consider $\Omega= (-1,0)\cup (0,1)$ and put $u(x) = -\frac12$ if $x\in (-1,0)$ and $u(x) = \frac12$ if $x\in [0,1)$. If we put $s=1-\eps$ then we have 
\begin{align*}
|u|^p_{W^{s,p}(\Omega)}= ps(1-s)\int_{0}^1\int_0^1\frac{\,\d y\,\d x}{(x+y)^{1+sp}} 
%= \frac{2}{sp}\int_{0}^1 x^{-sp}-(1+x)^{-sp}\,\d x
=
\begin{cases}
 \infty&\text{if } sp\geq 1,\\
 \frac{(1-s)}{1-sp} (2-2^{1-sp})&\text{if } sp< 1.
\end{cases}
\end{align*}
\begin{enumerate}[$(i)$]
\item\label{item-counter-ex1} Clearly, $u \in W^{1,p}(\Omega)$ for all $1\leq p<\infty$ with $\nabla u=0$ on $\Omega$. Note however that, the weak derivative of $u$ on $(-1,1)$ is $\delta_0$; the Dirac mass at the origin. It follows that $u\not \in W^{1,p}(-1,1)$ for all $1\leq p<\infty $ and $u\in BV(-1,1)$ with $|u|_{BV(-1,1)}=1.$
\item Moreover, $\Omega$ is not a $W^{1,p}$-extension domain. Indeed, assume $\overline{u}\in W^{1,p}(\R)$ is an extension of $u$ defined. In particular, $\overline{u}\in W^{1,p}(-1,1)$ and $\overline{u}= u$ on $\Omega$. The distributional derivative of $\overline{u}$ on $(-1, 1)$ is $\nabla \overline{u} = \delta_0$, This contradicts the fact that $\overline{u}\in W^{1,p}(\R)$. 

\item Since integrals disregard null sets, we have $\|u\|_{W^{s,p}(\Omega)} = \|u\|_{W^{s,p}(-1,1)}$ for all $1\leq p<\infty$. If $1<p<\infty$ and $s\geq 1/p$ then $\|u\|_{W^{s,p}(\Omega)} = \|u\|_{W^{s,p}(-1,1)} =\infty$ and hence $u\not \in W^{s,p}(\Omega)$. Thus the embedding $W^{1,p}(\Omega)\hookrightarrow W^{s,p}(\Omega)$ fails. However, if $p=1$ we get $u\in W^{s,1}(-1,1)$. Since $s=1-\eps$ this also implies that
\begin{align*}
	A_p= \liminf_{\eps\to 0}|u|^p_{W^p_{\nu_\eps}(\Omega)}=
	\begin{cases}
	\infty&\text{if $p>1$},\\
	1 &\text{if $p=1$}.
	\end{cases}
\end{align*}
\end{enumerate} 

%\vspace{2mm}
%B_1(0)\setminus \{(x',x_d)\in \R^d:~x_d=0\}
\noindent \textbf{Case $d\geq 2$.} The above example persists in higher dimension. Consider $\Omega$ be the unit ball $B_1(0)$ deprived with the hyperplane $\{x_d=0\}$ that is, $\Omega= B^+_1(0)\cup B^-_1(0)$ where $ B^\pm_1(0)= B_1(0)\cap \{(x',x_d)\in \R^d:\, \pm x_d>0\}$ and $u(x) = \frac{1}{2}\mathds{1}_{B^+_1(0)}(x)- \frac{1}{2}\mathds{1}_{B^-_1(0)}(x)$.
%	 Obviously, $u \in W^{1,p}(\Omega)$ for all $1\leq p<\infty$ with $\nabla u=0$. 
%Furthermore, following the discussion above, one can check that $u$ does not have any extension to the whole space. Hence $\Omega$ cannot be an extension domain. 
%On the other hand, $u$ does not belong to the fractional Sobolev space $W^{s,p}(\Omega)$ if $1<p<\infty$ and $s\geq 1/p$. Let us justify this by showing that $\|u\|_{W^{s,p}(\Omega)}=\infty$. 
Denoting balls in $\R^{d-1}$  as $B'_r(x')$, we have 
$\big\{(x,y)\in\R^d\times \R^d:\, x_d, y_d\in (0,1/2),\, x', y'-x'\in B'_{1/4}(0)\big\} \subset B_1^+(0)\times B_1^+(0).$
Enforcing the change of variables $y'=x'+ (x_d+y_d)h'$ so that $\d y'=(x_d+y_d)^{d-1}\d h'$ yields

\begin{align*}
|u|^p_{W^{s,p}(\Omega)}&=2a_{\eps, d,p} \int_{B_1^+(0)}\int_{B_1^+(0)}
\big(|x'-y'|^2+ (x_d+y_d)^2\big)^{-\frac{d+sp}{2}} \d y\,\d x\\
%&\geq a_{\eps, d,p}\iil_{D}\big(|x'-y'|^2+ (x_d+y_d)^2\big)^{-\frac{d+sp}{2}} \d y\,\d x\\
&\geq 2a_{\eps, d,p} \int_{0}^{1/2} \int_{0}^{1/2} \int_{B'_{1/4}(0)} 
\int_{B'_{1/4}(x')} \big(|x'-y'|^2+ (x_d+y_d)^2\big)^{-\frac{d+sp}{2}} \d y'\,\d x' \d y_d\,\d x_d\\
%&= 2a_{\eps, d,p}|B'_{1/4}(0)||\mathbb{S}^{d-2}|\int_{0}^{1/2} \int_{0}^{1/2}\int_{0}^{1/4}r^{d-2}\big(r^2+ (x_d+y_d)^2\big)^{-(d+sp)/2} \d r\,\d x_d\d y_d\\
%
%&= 2a_{\eps, d,p}|B'_{1/4}(0)||\mathbb{S}^{d-2}|\int_{0}^{1/2} \int_{0}^{1/2} (x_d+y_d)^{-1-sp} 
%\int_{0}^{1/4(x_d+ y_d)^{-1}}\hspace{-2ex}\frac{\rho^{d-2} \d \rho}{(1+\rho^2)^{(d+sp)/2}}\d x_d\d y_d\\
%
&= 2a_{\eps, d,p}|B'_{1/4}(0)|\int_{0}^{1/2} 
\int_{0}^{1/2}\int_{|h'|\leq \frac{1}{4(x_d+ y_d)}}\hspace{-0.0ex} 
\frac{\d h'}{(1+|h'|^2)^{\frac{d+sp}{2}} }\frac{\d x_d\d y_d}{(x_d+y_d)^{1+sp}}\\
\overset{\tfrac{1}{x_d+ y_d}\geq 1}{\geq} &
ps(1-s)\kappa^{1}_{d,p,s}\int_{0}^{1/2} \int_{0}^{1/2} 
\frac{\d x_d\d y_d}{(x_d+y_d)^{1+sp}},
 \qquad 
\kappa^{1}_{d,p,s}=2\frac{|B'_{1/4}(0)|}{|\mathbb{S}^{d-1}|} 
\int_{B'_{1/4}(0)}\frac{\d h'}{(1+|h'|^2)^{\frac{d+sp}{2}}}. 
\end{align*}

\noindent Analogously, since $B_1^+(0)\times B_1^+(0)\subset \big\{(x,y)\in\R^d\times \R^d:\, x_d, y_d\in (0,1),\,\, x'\in B'_{1}(0)\big\}$ we have 
\begin{align*}
|u|^p_{W^{s,p}(\Omega)}&=2a_{\eps, d,p} \int_{B_1^+(0)}
\int_{B_1^+(0)}\big(|x'-y'|^2+ (x_d+y_d)^2\big)^{-\frac{d+sp}{2}} \d y\,\d x\\
%&\leq a_{\eps, d,p}\iil_{D}\big(|x'-y'|^2+ (x_d+y_d)^2\big)^{-\frac{d+sp}{2}} \d y\,\d x\\
&\leq 2a_{\eps, d,p} \int_{0}^{1} \int_{0}^{1} \int_{B'_{1}(0)} \int_{\R^{d-1}} 
\big(|x'-y'|^2+ (x_d+y_d)^2\big)^{-\frac{d+sp}{2}} \d y'\,\d x' \d y_d\,\d x_d\\
%&= 2a_{\eps, d,p}|B'_{1}(0)||\mathbb{S}^{d-2}|\int_{0}^{1} \int_{0}^{1}\int_{0}^{\infty}r^{d-2}
%\big(r^2+ (x_d+y_d)^2\big)^{-\frac{d+sp}{2}} \d r\,\d x_d\d y_d\\
%
%&= 2a_{\eps, d,p}|B'_{1/4}(0)||\mathbb{S}^{d-2}|\int_{0}^{1} \int_{0}^{1} (x_d+y_d)^{-1-sp} 
%\int_{0}^{\infty}\hspace{-2ex}\frac{\rho^{d-2} \d \rho}{(1+\rho^2)^{\frac{d+sp}{2}}}\d x_d\d y_d\\
%
%&=2a_{\eps, d,p}|B'_{1}(0)|\int_{0}^{1} \int_{0}^{1} \frac{\d x_d\d y_d}{(x_d+y_d)^{1+sp}}
%\int_{\R^{d-1}}\hspace{-1ex}(1+|h'|^2)^{\frac{d+sp}{2}} \d h'\\
%
&= ps(1-s) \kappa^{2}_{d,p,s}\int_{0}^{1} \int_{0}^{1} 
\frac{\d x_d\d y_d}{(x_d+y_d)^{1+sp}}, \qquad \kappa^{2}_{d,p,s}= 2\frac{|B'_{1}(0)|}{|\mathbb{S}^{d-1}|}
\int_{\R^{d-1}} \frac{\d h'}{(1+|h'|^2)^{\frac{d+sp}{2}} }. 
\end{align*}

\noindent Note that $\kappa^{i}_{d,p,1}\leq \kappa^{i}_{d,p,s} \leq \kappa^{i}_{d,p,0}$. Using the case $d=1$ we draw the following conclusion, 
\begin{align*}
	|u|^p_{W^p_{\nu_\eps}(\Omega)}= |u|^p_{W^{s,p}(\Omega)}&\asymp
	\begin{cases}
	\infty&\text{if } sp\geq 1,\\
	\frac{(1-s)}{1-sp} (2-2^{1-sp})&\text{if } sp< 1.
	\end{cases}
\end{align*}

\begin{enumerate}[$(i)$]
	\item\label{item-counter-ex1a} Clearly, $u \in W^{1,p}(\Omega)$ and $u\not \in W^{1,p}(B_1(0))$ for all $1\leq p<\infty $. However $u\in BV(B_1(0))$. 
	
	\item \label{item-counter-exb}Moreover, $(i)$ implies that $\Omega$ is not a $W^{1,p}$-extension domain for $1\leq p<\infty $.
%	 Indeed, assume $\overline{u}\in W^{1,p}(\R)$ is an extension of $u$ defined. In particular, $\overline{u}\in W^{1,p}(-1,1)$ and $\overline{u}= u$ on $\Omega$. The distributional derivative of $\overline{u}$ on $(-1, 1)$ is $\nabla \overline{u} = \delta_0$, This contradicts the fact that $\overline{u}\in W^{1,p}(\R)$. 
	
	\item \label{item-counter-exc}As integrals disregard null sets, we have $\|u\|_{W^{s,p}(\Omega)} = \|u\|_{W^{s,p}(B_1(0))} =\infty$ for $1<p<\infty$ and $s\geq 1/p$ and hence $u\not \in W^{s,p}(\Omega)$. Thus the embedding $W^{1,p}(\Omega)\hookrightarrow W^{s,p}(\Omega)$ fails. However, if $p=1$ we get $u\in W^{s,1}(B_1(0))$. Furthermore, we have 
	\begin{align*}
	A_p= \liminf_{\eps\to 0}|u|^p_{W^p_{\nu_\eps}(\Omega)}\asymp
	\begin{cases}
	\infty&\text{if $p>1$},\\
1&\text{if $p=1$}.
	\end{cases}
	\end{align*}
\end{enumerate}
\end{counterexample}

\begin{proposition}\label{prop:cost-Kdp}
For any $e\in \mathbb{S}^{d-1}$ we have
\begin{align*}
	K_{d,p}= \fint_{\mathbb{S}^{d-1}} |w\cdot e|^p\d\sigma_{d-1}(w) = \frac{\Gamma\big(\frac{d}{2}\big)\Gamma\big(\frac{p+1}{2}\big)}{\Gamma\big(\frac{d+p}{2}\big) \Gamma\big(\frac{1}{2}\big)}.
\end{align*}
\end{proposition}

\begin{proof}
 The case $d=1$ is obvious and we only prove for $d\geq 2$. Since $ K_{d,p}$ is independent of $e\in \mathbb{S}^{d-1}$, it is sufficient to take $e=(0, \cdots, 0,1)$. Let $w = (w', t)\in \mathbb{S}^{d-1} $ with $t\in (-1,1)$ so that $w'\in \sqrt{1-t^2}\mathbb{S}^{d-2}$. The Jacobian for spherical coordinates gives $ \d \sigma_{d-1}(w)=\frac{ \d\sigma_{d-2}(w')dt}{\sqrt{1-t^2}}$ (see \cite[Appendix D.2]{grafakos04}). 
 Therefore, noting $ |\mathbb{S}^{d-1}|= \omega_{d-1}$, we have
\begin{align*}
K_{d,p}&= \fint_{\mathbb{S}^{d-1}} |w_d|^pd\sigma_{d-1}(w) 
&&=\frac{1}{\omega_{d-1}} \int_{-1}^{1}\int_{\sqrt{1-t^2}\mathbb{S}^{d-2}}\hspace{-2ex} |t|^p
\frac{\d\sigma_{d-2}(w')\d t}{\sqrt{1-t^2}} \\
&=\frac{2}{\omega_{d-1}}\int_{0}^{1} t^p\Big|\sqrt{1-t^2}\mathbb{S}^{d-2}\Big| 
\tfrac{ dt}{\sqrt{1-t^2}}
&&=\frac{2\omega_{d-2}}{\omega_{d-1}} \int_{0}^{1} (1-t^2)^{\frac{d-3}{2}} t^p dt \\
&=\frac{\omega_{d-2}}{\omega_{d-1}} \int_{0}^{1} (1-t)^{\frac{d-1}{2}-1} t^{\frac{p+1}{2}-1} dt 
&&= \frac{\omega_{d-2}}{\omega_{d-1}} B\big( \frac{d-1}{2}, \frac{p+1}{2}\big)= \frac{\omega_{d-2}}{\omega_{d-1}}\frac{\Gamma\left(\frac{d-1}{2}\right)\Gamma\left(\frac{p+1}{2}\right)}{\Gamma\left(\frac{d+p}{2}\right)}.
\end{align*}
%\vspace{-2mm}
Here $B(x,y):= \int_0^1 (1-t)^{x-1}t^{y-1}\d t, x>0, y>0$ is the beta function which links to the Gamma function by the relation $B(x,y)\Gamma(x+y)=\Gamma(x)\Gamma(y)$. 
The claim follows by using the formula
$ \omega_{d-1}= \frac{2\pi^{d/2}}{\Gamma\big(d/2\big)}$ along with $\Gamma(\frac{1}{2}) = \pi^{1/2}$. 
 %\pagebreak[2]
\end{proof}

\vspace{-2mm}
%\noindent \textbf{Data Availability Statement (DAS)}: Data sharing not applicable, no datasets were generated or analyzed during the current study.

\bibliographystyle{alpha}
%\bibliography{template-bibliography}

\begin{thebibliography}{DROV17}

\bibitem[AFP00]{AFP00}
Luigi Ambrosio, Nicola Fusco, and Diego Pallara.
\newblock {\em Functions of bounded variation and free discontinuity problems},
  volume 254 of {\em Oxford Mathematical Monographs}.
\newblock The Clarendon Press, Oxford University Press, New York, 2000.

\bibitem[App09]{App09}
David Applebaum.
\newblock {\em L\'{e}vy processes and stochastic calculus}, volume 116 of {\em
  Cambridge Studies in Advanced Mathematics}.
\newblock Cambridge University Press, Cambridge, second edition, 2009.

\bibitem[BBM01]{BBM01}
Jean Bourgain, Haim Brezis, and Petru Mironescu.
\newblock Another look at {S}obolev spaces.
\newblock In {\em Optimal control and partial differential equations}, pages
  439--455. IOS, Amsterdam, 2001.

\bibitem[Ber96]{Ber96}
Jean Bertoin.
\newblock {\em L\'{e}vy processes}, volume 121 of {\em Cambridge Tracts in
  Mathematics}.
\newblock Cambridge University Press, Cambridge, 1996.

\bibitem[BMR20]{BMR20}
Kaushik Bal, Kaushik Mohanta, and Prosenjit Roy.
\newblock Bourgain-{B}rezis-{M}ironescu domains.
\newblock {\em Nonlinear Analysis}, 199:111928, 10, 2020.

\bibitem[Bra18]{Bra18}
Julien Brasseur.
\newblock A {B}ourgain-{B}rezis-{M}ironescu characterization of higher order
  {B}esov-{N}ikol'skii spaces.
\newblock {\em Ann. Inst. Fourier (Grenoble)}, 68(4):1671--1714, 2018.

\bibitem[Bre02]{Brezis-const-function}
Haim Brezis.
\newblock How to recognize constant functions. {A} connection with {S}obolev
  spaces.
\newblock {\em Uspekhi Mat. Nauk}, 57(4(346)):59--74, 2002.

\bibitem[Bre11]{Bre10}
Haim Brezis.
\newblock {\em Functional analysis, {S}obolev spaces and partial differential
  equations}.
\newblock Universitext. Springer, New York, 2011.

\bibitem[D{\'a}v02]{Dav02}
Juan D{\'a}vila.
\newblock On an open question about functions of bounded variation.
\newblock {\em Calc. Var. Partial Differential Equations}, 15(4):519--527,
  2002.

\bibitem[DFK22]{DFK22}
{Jean-Daniel} Djida, {Guy Fabrice} {Foghem Gounoue}, and Yannick {Kouakep
  Tchaptchi{\'e}}.
\newblock Nonlocal complement value problem for a global in time parabolic
  equation.
\newblock {\em J. Elliptic Parabol. Equ.}, 8(2):767--789, 2022.

\bibitem[DNPV12]{Hitchhiker}
Eleonora Di~Nezza, Giampiero Palatucci, and Enrico Valdinoci.
\newblock Hitchhiker's guide to the fractional {S}obolev spaces.
\newblock {\em Bull. Sci. Math.}, 136(5):521--573, 2012.

\bibitem[DROV17]{DROV17}
Serena Dipierro, Xavier Ros-Oton, and Enrico Valdinoci.
\newblock Nonlocal problems with {N}eumann boundary conditions.
\newblock {\em Rev. Mat. Iberoam.}, 33(2):377--416, 2017.

\bibitem[EG15]{EG15}
Lawrence~C. Evans and Ronald~F. Gariepy.
\newblock {\em Measure theory and fine properties of functions}.
\newblock Textbooks in Mathematics. Chapman and Hall/CRC Press, Boca Raton, FL,
  revised edition, 2015.

\bibitem[FK22]{FK22}
Guy Foghem and Moritz Kassmann.
\newblock {A general framework for nonlocal Neumann problems}.
\newblock {arXiv e-prints: ~\url{https://arxiv.org/abs/2204.06793}}, 2022.

\bibitem[FKV15]{FKV15}
Matthieu Felsinger, Moritz Kassmann, and Paul Voigt.
\newblock The {D}irichlet problem for nonlocal operators.
\newblock {\em Math. Z.}, 279(3-4):779--809, 2015.

\bibitem[FKV20]{FGKV19}
{Guy Fabrice} {Foghem Gounoue}, Moritz Kassmann, and Paul Voigt.
\newblock Mosco convergence of nonlocal to local quadratic forms.
\newblock {\em Nonlinear Analysis}, 193:111504, 22, 2020.

\bibitem[{Fog}20]{Fog20}
{Guy Fabrice} {Foghem Gounoue}.
\newblock {\em {$L^2$-theory for nonlocal operators on domains}}.
\newblock PhD thesis, Bielefeld University,
  \url{https://doi.org/10.4119/unibi/2946033}, 2020.

\bibitem[Fog21]{Fog21b}
Guy Foghem.
\newblock {Nonlocal Gagliardo-Nirenberg-Sobolev type inequality}.
\newblock Preprint:\,\url{https://doi.org/10.48550/arXiv.2105.07989}, 2021.

\bibitem[GBR22]{GR22}
Miguel Garc\'{\i}a-Bravo and Tapio Rajala.
\newblock Strong {$BV$}-extension and {$W^{1,1}$}-extension domains.
\newblock {\em J. Funct. Anal.}, 283(10):Paper No. 109665, 39, 2022.

\bibitem[Gra14]{grafakos04}
Loukas Grafakos.
\newblock {\em Classical {F}ourier analysis}, volume 249 of {\em Graduate Texts
  in Mathematics}.
\newblock Springer, New York, third edition, 2014.

\bibitem[HKT08]{HKT08}
Piotr Haj{\l}asz, Pekka Koskela, and Heli Tuominen.
\newblock Sobolev embeddings, extensions and measure density condition.
\newblock {\em J. Funct. Anal.}, 254(5):1217--1234, 2008.

\bibitem[IN10]{IN10}
Hitoshi Ishii and Gou Nakamura.
\newblock A class of integral equations and approximation of {$p$}-{L}aplace
  equations.
\newblock {\em Calc. Var. Partial Differential Equations}, 37(3-4):485--522,
  2010.

\bibitem[KMS10]{KMS10}
Pekka Koskela, Michele Miranda, Jr., and Nageswari Shanmugalingam.
\newblock Geometric properties of planar {$BV$}-extension domains.
\newblock In {\em Around the research of {V}ladimir {M}az'ya. {I}}, volume~11
  of {\em Int. Math. Ser. (N. Y.)}, pages 255--272. Springer, New York, 2010.

\bibitem[Lah15]{Lah15}
Panu Lahti.
\newblock Extensions and traces of functions of bounded variation on metric
  spaces.
\newblock {\em J. Math. Anal. Appl.}, 423(1):521--537, 2015.

\bibitem[Leo17]{Leo17}
Giovanni Leoni.
\newblock {\em A first course in {S}obolev spaces}, volume 181 of {\em Graduate
  Studies in Mathematics}.
\newblock American Mathematical Society, Providence, RI, second edition, 2017.

\bibitem[LS11]{LS11}
Giovanni Leoni and Daniel Spector.
\newblock Characterization of {S}obolev and {$BV$} spaces.
\newblock {\em J. Funct. Anal.}, 261(10):2926--2958, 2011.

\bibitem[Lud14]{Lud14}
Monika Ludwig.
\newblock Anisotropic fractional {S}obolev norms.
\newblock {\em Adv. Math.}, 252:150--157, 2014.

\bibitem[MS02]{MS02}
Vladimir. Maz'ya and T.~Shaposhnikova.
\newblock On the {B}ourgain, {B}rezis, and {M}ironescu theorem concerning
  limiting embeddings of fractional {S}obolev spaces.
\newblock {\em Journal of Functional Analysis}, 195(2):230--238, 2002.

\bibitem[Pon04a]{Pon04}
Augusto~C. Ponce.
\newblock An estimate in the spirit of {P}oincar\'{e}'s inequality.
\newblock {\em J. Eur. Math. Soc. (JEMS)}, 6(1):1--15, 2004.

\bibitem[Pon04b]{Pon04-gamma}
Augusto~C. Ponce.
\newblock A new approach to {S}obolev spaces and connections to
  {$\Gamma$}-convergence.
\newblock {\em Calc. Var. Partial Differential Equations}, 19(3):229--255,
  2004.

\bibitem[PS17]{PS17}
Augusto~C. Ponce and Daniel Spector.
\newblock On formulae decoupling the total variation of {BV} functions.
\newblock {\em Nonlinear Anal.}, 154:241--257, 2017.

\bibitem[RO16]{Ros16}
Xavier Ros-Oton.
\newblock Nonlocal elliptic equations in bounded domains: a survey.
\newblock {\em Publ. Mat.}, 60(1):3--26, 2016.

\bibitem[Sat13]{Sat13}
Ken-iti Sato.
\newblock {\em L\'{e}vy processes and infinitely divisible distributions},
  volume~68 of {\em Cambridge Studies in Advanced Mathematics}.
\newblock Cambridge University Press, Cambridge, 2013.
\newblock Translated from the 1990 Japanese original, Revised edition of the
  1999 English translation.

\bibitem[Wil91]{Wil91}
David Williams.
\newblock {\em Probability with martingales}.
\newblock Cambridge Mathematical Textbooks. Cambridge University Press,
  Cambridge, 1991.

\bibitem[Zho15]{Zh15}
Yuan Zhou.
\newblock Fractional {S}obolev extension and imbedding.
\newblock {\em Trans. Amer. Math. Soc.}, 367(2):959--979, 2015.

\bibitem[Zor16]{Zor16}
Vladimir~A. Zorich.
\newblock {\em Mathematical analysis. {II}}.
\newblock Universitext. Springer, Heidelberg, second edition, 2016.
\newblock Translated from the fourth and the sixth corrected (2012) Russian
  editions by Roger Cooke and Octavio Paniagua T.

\end{thebibliography}

\end{document}